\documentclass[10pt]{amsart}

\usepackage[foot]{amsaddr}

\title[On the number of defects in optimal quantizers]{On the number of defects in optimal quantizers on closed surfaces: the hexagonal torus}

\author{Jack Edward Tisdell$^\dagger$}
\email{jack.tisdell@mail.mcgill.ca}
\author{Rustum Choksi$^\dagger$}
\email{rustum.choksi@mcgill.ca}
\address{$^\dagger$Department of Mathematics and Statistics, McGill University}
\author{Xin Yang Lu$^\ddagger$}
\address{$^\ddagger$Department of Mathematical Sciences, Lakehead University}
\email{xlu8@lakeheadu.ca}
\date{\today}

\usepackage{enumitem}
\usepackage{amssymb}
\usepackage{mathrsfs}
\usepackage{mathtools}
\usepackage{tikz}
\usetikzlibrary{calc, quotes, angles, decorations.markings, decorations.pathreplacing, arrows.meta}
\usepackage{xcolor}

\newcommand\T{\mathbb T}

\newcommand\R{\mathbb R}

\newcommand\N{\mathbb N}

\newcommand\E{\mathcal E}
\newcommand\M{\mathcal M}
\newcommand\m{\mu}
\newcommand\s{s}

\newcommand\Lo{\mathscr L}

\DeclareMathOperator\gap{gap}
\DeclareMathOperator\maxgap{max\,gap}
\DeclareMathOperator\df{def}
\DeclareMathOperator\diam{diam}
\DeclarePairedDelimiter\abs\lvert\rvert
\DeclarePairedDelimiter\norm\lVert\rVert

\newtheorem{theorem}{Theorem}[section]
\newtheorem{lemma}[theorem]{Lemma}
\newtheorem{corollary}[theorem]{Corollary}
\theoremstyle{definition}
\newtheorem*{definition}{Definition}
\newtheorem*{question}{Question}
\theoremstyle{remark}
\newtheorem{remark}{Remark}[section]

\usepackage[]{hyperref}

\begin{document}
\begin{abstract}
    We present a strategy for proving an asymptotic upper bound on the number of defects (non-hexagonal Voronoi cells) in the $n$ generator optimal quantizer on a closed surface (i.e., compact 2-manifold without boundary). The program is based upon a general lower bound on the optimal quantization error and related upper bounds for the L\"oschian numbers $n$ (the norms of the Eisenstein integers) based upon the Goldberg-Coxeter construction. A gap lemma is used to reduce the asymptotics of the number of defects to precisely the asymptotics for the gaps between L\"oschian numbers.
     We apply this strategy on the hexagonal torus and prove that the number of defects is at most $O(n^{1/4})$---strictly fewer than surfaces with boundary---and conjecture (based upon the number-theoretic L\"oschian gap conjecture) that it is in fact $O(\log n)$. Incidentally, the method also yields a related upper bound on the variance of the areas of the Voronoi cells. 
     We show further that the bound on the number of defects holds in a neighborhood of the optimizers. Finally, we remark on the remaining issues for implementation on the 2-sphere.
\end{abstract}

\maketitle

\section{Introduction}
Finding the optimal way to quantize a continuous distribution over a domain or manifold into a fixed number of points falls under the broad class of problems referred to as \emph{Optimal Quantization}. These types of problems tend to be highly non-local and are related to \emph{the crystallization conjecture} in mathematical physics  (\cite{Blanc2015, Bourne-Cristoferi2021}). Our interest here entails the minimization of energy functionals (quantization errors)  which exhibit a type of \emph{semi-locality} as they can be defined in terms of the ``neighborhoods'' of each point, for some suitable notion of ``neighborhood''. A natural approach along these lines is to take as the total energy the sum of individual energies of the Voronoi cells of each point in the tessellation they generate. Indeed, this is precisely what one obtains whenever one defines an energy density (whose integral is the energy) that depends only on the \emph{nearest} point---an obvious attempt toward semi-locality. 
To this end, we investigate for each $n$, the variational problem
\begin{equation}
    \text{minimize } 
    \E(Y) = \int_\M \min_{y\in Y} \varrho_\M(x,y)^r\,d\sigma_\M(x)
    \text{ over all $Y \subset \M$ with $\abs Y \le n$},
    \label{eq:energy}
\end{equation}
where $\M$ is some compact surface (a Riemannian $2$-manifold, to be precise), $\varrho_\M$ and $\sigma_\M$ are its induced distance and surface measure, and $r > 0$. We refer to $\E(Y)$ as the \emph{quantization error} (with respect to exponent $r$) in quantizing the measure $d\sigma_\M(x)$ by the discrete set $Y$. 
As mentioned above, this defining integral splits into the sum over Voronoi cells: 
\[
    \E(Y) 
    = \int_\M \min_y \varrho_\M(x,y)^r\,d\sigma_\M(x) 
    = \sum_y \int_{D_y} \varrho_\M(x,y)^r\,d\sigma_\M(x)
\]
where $D_y$ is the Voronoi cell containing $y$ in the diagram generated by $Y$. Thus, the Voronoi diagrams of solutions and near solutions are of central interest for this class of problems.
In this context, optimal quantization generalizes the problem of finding the \emph{optimal centroidal Voronoi tessellation} when $r=2$ (\cite{Du_1999}). These problems are also directly linked to semi-discrete optimal transport (see for example, \cite{Bourne2024}). 

There has been significant work, particularly in two dimensions, on the asymptotic behaviour of solutions as the number of generating points $n$ tends to infinity. First off, in Euclidean domains a well-known conjecture attributed to Gersho~\cite{Gersho1979} addresses the periodic nature of the configuration with least quantization error (alternatively, the centroidal Voronoi tessellation with lowest energy). It asserts that there exists a polytope $V$ (not depending on $r$) which tiles Euclidean space such that all interior Voronoi cells in the optimal configuration are asymptotically congruent scaled copies of $V$. In two dimensions (where the conjecture is fully resolved), the optimal polytope $V$ is a regular hexagon, corresponding to an optimal placement of points on a triangular lattice. In dimension three, the optimal polytope $V$ is conjectured to be the truncated octahedron,  corresponding to an optimal placement of points on a BCC (body centered cubic) lattice (see for example, \cite{Barnes1983,  Du2005, Choksi-Lu}). Voronoi cells which include the boundary are of course irregular, the number of which in two dimensions is $O(\sqrt n)$. 

Thanks to the seminal work of Gruber, the asymptotic emergence of the regular hexagonal tiling in two dimensions is totally generic. Indeed, for essentially any compact Riemannian 2-manifold and a broad class of functions $f$, the Voronoi cells tend toward nearly congruent regular hexagons (asymptotically Euclidean) except for a \emph{vanishingly small proportion}---i.e., $o(n)$ many---of possibly non-hexagonal cells. We refer the reader to Gruber~\cite{Gruber2001, Gruber2004} for the fully general results, but we will recall in Section \ref{sec-Gruber} the special cases that apply to our problem. 

In this article, we consider domains which are compact 2-manifolds without boundary and address the asymptotics for the number of non-hexagons in the optimal configuration. Borrowing from crystallography (see for example, \cite{Brojan2014, Jimenez2016, Guerra2018, Praetorius2018, Klatt2019, Meyra2019}), let us make the following definition: 
\begin{definition}
    A (topological) {\bf defect} in a quantizer $Y$ is a non-hexagonal Voronoi cell. We denote the number of such defects by $\df(Y)$. 
\end{definition}
We remark that from a mathematical point of view, the issue of defects in the optimal quantizers (and in crystallography in general) is largely open. Clearly, the $O(\sqrt n)$ many defects in the planar Euclidean case is a tighter bound than Gruber's general $o(n)$. It's also the optimal result in that case, the solutions are always well-distributed throughout the domain and so there must be, up to a constant factor, at least $\sqrt n$ many defects due to boundary effects alone. This invites the main question:
\begin{question}
    On a \emph{closed} surface, do energy minimizers have strictly fewer than $O(\sqrt n)$ many defects and if so, how much fewer?
\end{question}

Here we present a strategy to answer this question in the affirmative, and moreover produce quantitative upper bounds on the number of defects. The program is based upon a general lower bound on the optimal quantization error (the energy) and related upper bounds for  L\"oschian numbers $n$ (the norms of the Eisenstein integers) realized using the Goldberg-Coxeter construction. A gap lemma is used to reduce the asymptotics of the number of defects (for arbitrary $n$) to precisely the asymptotics for the gaps between L\"oschian numbers. 
Whereas, our ultimate goal is the application on the 2-sphere, here we complete the program for the \emph{hexagonal torus}; the simplest closed surface for which both Gruber's results and the Goldberg-Coxeter construction apply directly. It's important to understand that, nonetheless, the problem is non-trivial on the hexagonal torus for this surface does not admit regular hexagonal tilings except for L\"oschian numbers $n$, while the asymptotically optimal quantization problem concerns arbitrary (large) $n$.
We prove that the number of defects is at most $O(n^{1/4})$---strictly fewer than surfaces with boundary---and conjecture (based upon the number-theoretic L\"oschian gap conjecture) that it is in fact $O(\log n)$. Our  method also yields a related upper bound on the variance of the areas of the Voronoi cells. We show further that the bound on the number of defects holds in a neighborhood of the optimizers.

All of our results concern arbitrary positive exponent $r > 0$. Generality aside, we feel that this choice is insightful even for readers who are only interested in $r=2$, say. Especially when $r = 2$, one can easily lose sight of the distinct effects due to the value of $r$, the 2-dimensionality of the surface, and second order expansions of various functions. The authors found the general $r > 0$ perspective enormously helpful in teasing these apart.

Before proceeding, we wish to underscore and address a possible source of confusion. Namely, the \emph{number of defects of the minimal energy $n$-point configuration} is not to be misconstrued with the \emph{minimal number of defects among all $n$-point configurations}. In fact, for some manifolds $\M$ (including flat tori), it can be shown that there is a configuration for every sufficiently large $n$ whose Voronoi tessellation consists \emph{entirely} of hexagons, hence has no defects. The point is that even if a surface admits configurations for all (large) $n$ whose tessellations have no defects, their hexagons may be far from regular. But meanwhile, due to the extremely general results of Gruber, there is a sense in which the energy minimizers tend towards configurations whose tessellations are almost regular hexagonal tilings. It is this \emph{tension} which introduces defects into the minimizers: it is energetically preferable for a configuration to have as many \emph{nearly regular} hexagons as possible at the expense possibly of some non-hexagons. Indeed, we observe this phenomenon numerically (for $r=2$). For random initializations in the hexagonal torus for fixed $n$, most algorithms almost invariably converge to defect configurations. But even a state-of-the-art algorithm, \emph{MACN} (see \cite{Gonzalez2021}), which is capable of finding the regular (global) minimizer for $n$ which permit it, converges to defect configurations for other $n$. An example, $n=320$ is shown in Figure~\ref{fig:macn}. 

\begin{figure}[ht]
    \centering
    \includegraphics{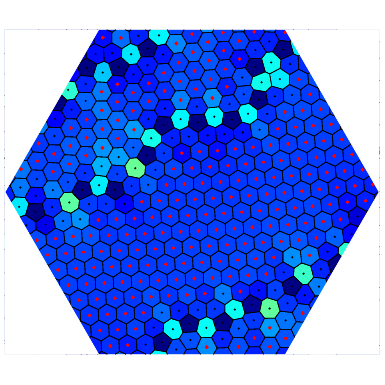}
    \caption{Best configuration for $n=320$ found in a run of MACN.}
    \label{fig:macn}
\end{figure}

As we have stated, our ultimate goal is to apply this program on the 2-sphere and obtain the same results based upon the L\"oschian gap conjecture. The remaining issue is the  upper bound constructions for the L\"oschian $n$ based upon suitable projections of certain Goldberg polyhedra onto the sphere. 
The structure of optimal constructions on the 2-sphere is a fascinating and largely open  field and includes such problems as Smale's 7th problem and the Thomson problem on the sphere (both unsolved in general). 

\subsection{Outline}
\label{sec:outline}
To demonstrate our main result, we will derive a lower bound on the energy minima by direct calculation. We will then provide an exactly corresponding upper bound \emph{along a subsequence} by way of an explicit construction (the Goldberg-Coxeter construction). Finally, using the fact that the energy minima are non-increasing in $n$, we will be able to interpolate between the subsequence of tight upper bounds while accounting for the error. We'll see that the upper bound we obtain for all $n$ still matches the lower bound to leading order and it's in the next-order difference that the defect structure is revealed. We will also discuss stability of the number of defects near energy minimizers. The paper is outlined as follows.
\begin{itemize}
    \item In Section~\ref{sec:toolkit}, we introduce some basic definitions as well as some key theoretical tools. 
        \begin{itemize}
            \item We state the special case of Gruber's results which apply in this setting and prove the following consequences. (1) The diameter and area of the Voronoi cells of the minimizers scale like $n^{-1/2}$ and $n^{-1}$, respectively, and there are uniform upper bounds on the number of cells neighboring any given one and on the number of cells meeting at any one point. (2) We show that in the flat case\footnote{The flat case suffices for the results of this paper but we expect this lemma to generalize readily, in particular to the constant curvature setting.}, the average number of Voronoi edges per cell tends to six as $n \to \infty$. 
            \item We define the \emph{regular polygon moment} $\m(a,k)$, a function which naturally arises in the leading order contribution to the energy minima, and discuss its relevant properties. 
            \item We provide the ``$\frac1n + O(e(n))$ interpolation lemma'', an elementary but paramount result. This result allows us later to fill in the gaps between the values of $n$ for which our construction gives explicit upper bounds and estimate the worst case error incurred in doing so.
        \end{itemize}
    \item Section~\ref{sec:torus} contains the main body of the argument. We focus on the hexagonal torus $\T$ and first compute a lower bound for the energy minima. We then explain the Goldberg-Coxeter construction and relate the resulting tilings of $\T$ to the energy, giving an upper bound on an explicit subsequence. Finally, the main theorem follows as a consequence of these bounds and the interpolation lemma. 
    \item In the short Section~\ref{sec:stability}, we discuss the stability of the main result of Section~\ref{sec:torus}, effectively extending it to a neighborhood of each minimizer whose size we quantify. The result we obtain is weaker than we believe is possible but it suffices to rule out certain otherwise plausible pathologies due to the inherently discontinuous nature of the defect counting map. It also serves to highlight quantitatively the sensitivity to perturbations of the various steps of reasoning in the prior sections.
    \item Finally, in Section~\ref{sec:sphere}, we discuss our work in progress adapting this argument to the 2-sphere. 
\end{itemize}

\subsection{Notation}
\label{sec:notation}
Throughout the paper, we adopt the following notational conventions. $\M$ always denotes a (compact orientable) Riemannian 2-manifold with distance $\varrho_\M : \M \times \M \to [0,\infty)$ and surface measure $\sigma = \sigma_\M$ both induced by the Riemannian metric. As in the previous sentence, we typically omit the subscript $\M$ on the surface measure. We typically denote points of $\M$ with lower-case Latin letters near the end of the alphabet $w,x,y,z,\ldots$, possibly with subscripts or other decorations $z_j,y'$, etc. Diameters of sets $E \subset \M$ and open balls in $\M$ are denoted by $\diam_\M(E) = \sup_{x,y \in E} \varrho_\M(x,y)$ and $B(x,r) = \{y \in \M \mid \varrho_\M(x,y) < r\}$, respectively. Following Section~\ref{sec:toolkit}, we are interested in the specific setting $\M = \T = S^1 \times \frac{\sqrt 3}{2}S^1$ of the hexagonal torus and later in more general flat tori $\T_\gamma = S^1 \times \gamma S^1$. 

$\E$ always denotes the functional defined in \eqref{eq:energy} with respect to the (implicit) exponent $r$. The letter $r$ always denotes the exponent appearing in $\E$. We will always denote by $Y_n \subset \M$ a solution of the minimization problem \eqref{eq:energy}. In particular, $Y_n$ is (without loss of generality) a set of exactly $n$ points. Sometimes, we'll employ similar notations, e.g., $Y, \tilde Y_n, Y_{p,q} \subset \M$, for finite sets of points not necessarily minimizers of $\E$ (but perhaps related to one $Y_n$). Very often when working with such a $Y_n$, we will be interested in sums and minima over $y \in Y_n$, which in context we shall denote $\sum_y \cdots$ and $\min_y \cdots$ without confusion. The letter $n$ is exclusively used in this way, the cardinality of the solution $Y_n$ under consideration. We are interested in large $n$ and asymptotic/convergence statements are understood as $n\to\infty$ unless otherwise stated. We always denote by $D_y \subset \M$ the (closed) Voronoi cell of $y \in Y_n$ in the Voronoi tessellation of $\M$ generated by $Y_n$ (as defined in the next section).

Several objects of central interest are the number of defects $\df(Y_n)$ in the Voronoi tessellation of $Y_n$, the variance $s^2(Y_n) = \frac1n \sum_y \abs{\sigma(D_y) - \frac1n \sigma(\M)}^2$ of the areas of the Voronoi cells, and the set $\Lo = \{p^2 + pq + q^2\mid p,q \in \N\}$ of L\"oschian numbers\footnote{\url{https://oeis.org/A003136}} (equivalently, the norms of the Eisenstein integers).

We adopt the standard definitions for Bachman-Landau ``big-O'' notations for asymptotic bounding and domination including $O$, $\Omega$, $\Theta$, and $o$. Namely, $f(n) = O(g(n))$ whenever there is a $C > 0$ such that $f(n) \le Cg(n)$ eventually; $f(n) = \Omega(g(n))$ if there is a $C > 0$ such that $f(n) \ge Cg(n)$ eventually; $f(n) = \Theta(g(n))$ if both $f(n) = O(g(n))$ and $f(n) = \Omega(g(n))$; and $f(n) = o(g(n))$ if for all $\varepsilon > 0$, eventually $f(n) \le \varepsilon g(n)$. Again, these are always understood as $n\to\infty$ unless otherwise stated. These quantities are non-negative unless otherwise specified, e.g., we will sometimes write $\pm O(e(n))$ to mean a quantity between $-Ce(n)$ and $+Ce(n)$ for some positive constant $C$. We also write $f(n) \sim g(n)$ to mean $\frac{f(n)}{g(n)} \to 1$ as $n \to \infty$ (note $f \sim g$ is strictly stronger than $f = \Theta(g)$).

\section{Our toolkit}
\label{sec:toolkit}
From a bird's eye view, so to speak, our arguments employ four essential tools: Gruber's results (and some consequences thereof) concerning solutions of the minimization problem \eqref{eq:energy}; a paramount function $\m(a,k)$ we call the regular polygon moment; a means of ``filling in the gaps'' between asymptotic estimates of non-increasing functions about which we have only partial information; and the well-known Goldberg-Coxeter construction. The Goldberg-Coxeter construction is best explained when we need it so we shall withhold its discussion until the next section. For now, we remark that the applicability of the Goldberg-Coxeter construction is what distinguishes the orientable equi-triangulated surfaces. The first three tools, on the other hand, are totally general and we will discuss them in turn after setting the stage.

Given a finite set of points $Y \subset \M$, the (closed) \emph{Voronoi cell $D_y$} of $y \in Y$ (with respect to $Y$) is
\[
    D_y = \{x \in \M : \text{$\varrho_\M(x,y) \le \varrho_\M(x,y')$ for all $y' \in Y$}\},
\]
that is, $D_y$ is the region in $\M$ consisting of all points closer to $y$ than to any other point of $Y$. A Voronoi edge is a region of overlap between two Voronoi cells and a Voronoi vertex is a point common to three or more cells. By the \emph{degree} of a Voronoi vertex, we mean the number of Voronoi cells incident to it. Distinct $y,y' \in Y$ whose cells share an edge are said to be Voronoi neighbors. In the flat torus, the Voronoi cells are always convex polygons and the notions of edge and vertex coincide with the polygonal notions. The Dirichlet-Voronoi cells tessellate the whole surface in the sense that their union is $\M$ and their interiors are disjoint. The tessellation as a whole is referred to as the \emph{Voronoi tessellation of $Y$}.

\subsection{Gruber's results on optimal configurations and some lemmas}\label{sec-Gruber}
We now recall some key results of Gruber concerning solutions of the minimization problem \eqref{eq:energy}. We refer readers to the cited works for the details and the more general statements.

\begin{theorem}[Gruber]
    Let $\E$ be the energy functional defined in \eqref{eq:energy} for any $r > 0$. Suppose $Y_n$ (with $\abs{Y_n} = n$) attains the minimum $\E(Y_n)$ for each $n = 1,2,\dots$. Then there are constants $C,C' > 0$ such that
    \begin{enumerate}
        \item ($\Omega(n^{-1/2})$ uniformly discrete.) $\varrho_\M(y,y') \ge C/\sqrt n$ for all distinct $y,y' \in Y_n$,
        \item ($O(n^{-1/2})$ covering radius.) $\min_{y \in Y_n}\varrho_\M(y,x) \le C'/\sqrt n$ for every $x \in \M$.
    \end{enumerate}
    We say that $Y_n$ satisfying the above two properties is $n^{-1/2}$-Delone. Furthermore, if $\tilde Y_n$ with $\abs{\tilde Y_n} = n$ satisfy $\E(\tilde Y_n) \sim \E(Y_n)$, then
    \begin{enumerate}[resume]
        \item\label{item:asym_reg_hex} the sequence $\tilde Y_n$ is asymptotically a regular hexagonal pattern\footnote{We acknowledge the strangeness of this definition. The $1.1$ is \emph{ad hoc}, any $1+\delta$ for sufficiently small $\delta$ would work just as well if at the expense of a greater constant in the $o(n)$ Landau symbol. In the interest of rigor, one might more seriously worry that this definition admits \emph{degenerate} ``regular hexagons'' with $y$ located at the corner of an equilateral triangle of side length $\varrho_n$ with $y_1,y_3,y_5$ near another corner and $y_2,y_4,y_6$ near the third. But even if there were such a $y$, then its $y_1,\dots,y_6$ would themselves fail to have the asymptotically regular hexagonal pattern property (even allowing for degeneracy) and so this kind of degeneracy can only happen $o(n)$ many times anyway and all such $y$ can be regarded as among the exceptions.} of edge length $\varrho_n = \sqrt{\frac{2\sigma(\M)}{\sqrt{3}n}}$ (to borrow Gruber's terminology) in the sense that there is a positive sequence $\varepsilon_n$ converging to $0$ such that for each $y \in \tilde Y_n$, with at most $o(n)$ many exceptions, there are six distinct $y_1,\dots,y_6 \in \tilde Y_n$ all distinct from $y$ such that
            \[
                B(y,1.1\varrho_n) \cap \tilde Y_n = \{y,y_1,\dots,y_6\}
            \]
            holds where 
            \[
                \varrho_\M(y,y_j), \varrho_\M(y_j,y_{j+1}) = (1\pm \varepsilon_n)\varrho_n\text{ for $j = 1,\dots,6$ and $y_7 = y_1$}.
            \]
            Here $(1\pm\varepsilon_n)\varrho_n$ denotes a quantity between $(1-\varepsilon_n)\varrho_n$ and $(1+\varepsilon_n)\varrho_n$. 
    \end{enumerate}
    \label{thm:Gruber}
\end{theorem}

\begin{proof}
    We refer the reader to~\cite[Theorem~1]{Gruber2004} for the Delone properties and to~\cite[Theorem 2]{Gruber2001} for the asymptotic regularity.
\end{proof}

Obviously, property~(\ref{item:asym_reg_hex}) applies to the minimizers $Y_n$ themselves and this will be enough until Section~\ref{sec:stability} when we study stability.

The $n^{-1/2}$-Delone property has some immediate consequences we wish to highlight.

\begin{corollary}
    Say $\M$ is compact and without boundary. Let $Y_n$ as in the previous theorem and let $D_y$ denote the Voronoi cell of $y\in Y_n$ with respect to $Y_n$. There exist constants $D_\pm, A_\pm, K, L > 0$ not depending on $n$ such that for every $y \in Y_n$
    \[
        \frac{D_-}{\sqrt n} \le \diam_\M(D_y) \le \frac{D_+}{\sqrt n},
        \qquad
        \frac{A_-}{n} \le \sigma_\M(D_y) \le \frac{A_+}{n}.
    \]
    Moreover, each $y \in Y_n$ has at most $K$ Voronoi neighbors and no more than $L$ Voronoi cells meet at any vertex.
    \label{cor:Delone}
\end{corollary}

\begin{remark}
    Because $\M$ is without boundary, the boundary of each Voronoi cell $D_y$ is just the union of its Voronoi edges and so there is no ambiguity in referring to the Voronoi edges of $D_y$ simply as edges. Furthermore, the number of edges of $D_y$ is, by definition, the number of Voronoi neighbors of $y$ in $Y_n$. Since we are only concerned here with surfaces without boundary, we will speak interchangeably of the number of Voronoi neighbors of a given point and the number of (Voronoi) edges of its cell without further comment.
\end{remark}

\begin{proof}
    The diameter bounds are trivial consequences of the $n^{-1/2}$-Delone properties of $Y_n$. The bounds on the area follow from these together with the compactness of $\M$ (assuring bounded curvature) and the fact that $\M$ has no boundary.\footnote{The argument works just as well for $\M$ with boundary so long as the boundary is sufficiently regular. In general, one must be careful about situations that affect the usual relationship between area/volume and radius of balls. E.g., if $\M$ is a region in the plane with piecewise smooth boundary but having a point where two boundary pieces meet tangentially.}

    For the upper bound on the number of edges let $C,C' > 0$ as in Theorem~\ref{thm:Gruber}, let $y \in Y_n$ and say $y_1,\dots,y_k$ are the Voronoi neighbors of $y$. We know the $y_j$ are at least $C/\sqrt n$ away from each other and so the balls $B(y_j,C/2\sqrt n)$ are all disjoint. But also $\varrho_\M(y,y_j) \le 2C'/\sqrt n$ by the triangle inequality (considering any point belonging to the edge $D_y \cap D_{y_j}$). Thus, all $k$ balls $B(y_j,C/2\sqrt n)$ are contained in the ball $B(y, \frac{2C'}{\sqrt n} + \frac{C}{2\sqrt n})$. Since $\M$ is compact and without boundary, there are constants $0 < \alpha \le 1 \le \beta$ such that any ball of sufficiently small radius $r$ has area between $\alpha\pi r^2$ and $\beta\pi r^2$. Thus, if $n$ is large enough,
    \[
        k \alpha\pi\frac{C^2}{4n}
        \le \sum_{j=1}^k\sigma_\M( B(y_j, \tfrac{C}{2\sqrt n}))
        \le \sigma_\M( B(y, \tfrac{2C'+C/2}{\sqrt n}))
        \le \frac{\beta\pi(2C'+C/2)^2}{n}
    \]
    so $k \le 4\beta(2C'+C/2)^2/\alpha C^2$.

    One argues similarly for the vertex degree bound, supposing $v \in \M$ is a common vertex of the cells of $y_1,\dots,y_l \in Y_n$ and considering appropriate balls.
\end{proof}

An intuitive, but less straightforward, consequence of Theorem~\ref{thm:Gruber} (which uses property~\ref{item:asym_reg_hex}) is the following.

\begin{lemma}
    Suppose $\M$ is compact, without boundary, and flat\footnote{The flatness assumption is not essential  but it makes the proof much more readable than it would be otherwise. In general, one works in Riemannian normal coordinates and in neighborhoods which are small enough that circles, angles, etc.\ are as close to Euclidean as required. By compactness, only finitely many charts are necessary and such neighborhoods are uniformly small across $\M$.} in the sense that it is locally isometric to the Euclidean plane. Let $Y_n$ as in the previous theorem. Then the average number of Voronoi edges per cell converges to $6$ as $n \to \infty$ and is at most $6\big(1 - \frac\chi n\big)$ for all $n$ where $\chi$ is the Euler characteristic of $\M$.
    \label{lemma:hexagon_average}
\end{lemma}

\begin{remark}
    If one is willing to distort $Y_n$, say to $\tilde Y_n$, then the result trivially follows for $\tilde Y_n$ by Euler's polytope formula because one can ensure, while keeping the distortion within any desired tolerance, that the Voronoi tessellation of $\tilde Y_n$ only has degree 3 vertices. Indeed, this kind of approach is extremely common when working with Voronoi tessellations, however, such a move is not suitable here because the central quantity of interest, the number of defects in the tessellation $\df(Y_n)$, might be sensitive even to arbitrarily small distortions in case $Y_n$ so happens to be rather unfortunate, a possibility we cannot \emph{prima facie} rule out.
\end{remark}

\begin{proof}
    The upper bound of $6 - \frac\chi n$ follows immediately from Euler's polytope formula. 
    Let $\chi$ be the Euler characteristic of $\M$, let $v,e$ be the total number of Voronoi vertices and edges, respectively, let $k_y$ be the number of vertices of of the cell $D_y$, and let $\bar k = \frac1n\sum_y k_y$ be the average. Since at least three cells meet at each vertex, $3v \le \sum_y k_y = n\bar k$. By Euler,
    \[
        \chi 
        = v-e+n
        \le \frac13 n\bar k - \frac12 n\bar k + n = -\frac16 n\bar k + n
    \]
    and so $\bar k \le 6\big(1 - \frac\chi n\big)$. 

    To obtain a lower bound which will show the convergence to $6$, we need to work harder.
    Say a Voronoi vertex is of high degree if it has degree $>3$ and say a point $y \in Y_n$ is \emph{nice} if it witnesses the regularity property (3) of Theorem~\ref{thm:Gruber}. The idea is to show that at least one of the cells incident to any high-degree vertex must not be nice (intuitively, because it has an acute internal angle at that vertex). Gruber's result is that all but $o(n)$ many points are nice, and so we'll have that at most $o(n)$ many Voronoi vertices (out of $\Omega(n)$ many total vertices) are high degree. The lemma then follows from an Euler's polytope formula argument.

    Let $v_d$ be the total number of vertices of degree $d$ and let $v^+ = \sum_{d\ge 4} v_d = v-v_3$ be the number of high degree vertices. To make the proof more readable, we have suppressed the dependence on $n$ in the notations $e,v,v_d,v^+,\bar k$, but bear it in mind. 

    First, let's see that if $v^+/v \to 0$ as $n \to \infty$, then the conclusion of the lemma follows.
    By the previous corollary, there is a uniform upper bound $L$ on the vertex degree for all $n$. Assuming $v^+/v \to 0$, take $\gamma >0$ as small as desired and $n$ large enough that $v^+/v \le \gamma/L$. Then
    \[
        n\bar k
        = \sum_y k_y
        = \sum_{d\ge 3} dv_d
        = 3v_3 + \sum_{d\ge4} dv_d
        \le 3v + Lv^+
        \le (3+\gamma)v.
    \]
    Applying Euler's polytope formula, we find $\chi = v -e + n \ge \frac1{3+\gamma}n\bar k - \frac12 n\bar k + n = -n\bar k\frac{1+\gamma}{6+2\gamma} + n$ and finally $\bar k \ge \frac{6 + 2\gamma}{1+\gamma}\big(1 - \frac\chi n\big) = \big(6 - \frac{4\gamma}{1+\gamma}\big)\big(1 - \frac\chi n\big)$. Together with the upper bound $\bar k \le 6\big(1 - \frac\chi n\big)$, we have $\bar k \to 6$ as $n \to \infty$. 

    So we need to show that $v^+/v \to 0$, i.e., the proportion of high degree vertices tends to zero. Observe that $Lv \ge \sum_y k_y \ge 3n$ and so $v \ge \frac{3n}{L} = \Omega(n)$, hence it suffices to show $v^+ = o(n)$. 

    By Corollary~\ref{cor:Delone}, we may assume that $n^{-1/2}$ is much less than the diameter of $\M$ so that, in particular, each Voronoi vertex and its neighboring points of $Y_n$ lie in a small region isometric to the Euclidean plane. Suppose $w \in \M$ is a vertex of degree $\ge 4$. Then, by definition, there are $y_1^w,\dots,y_l^w \in Y_n$ for $l \ge 4$ all equidistant from $w$, say at distance $R_w = \varrho_\M(w,y_i^w)$, and there are no points of $Y_n$ strictly closer to $w$. Let $C_w$ be the (Euclidean) circle centered at $w$ of radius $R_w$, thus passing through the $y^w_i$. Say that $y_1^w,\dots, y_l^w$ are numbered in order as they appear around one circuit of $C_w$. They partition $C_w$ into $l$ disjoint arcs, say $s_i$ is the length of the arc joining $y_i^w$ to $y_{i+1}^w$ and assume, without loss of generality, that $s_1$ is the least among $s_1,\dots,s_l$. In particular, $s_1$ cannot exceed a quarter the circumference of $C_w$ and so $\varrho_\M(y_1^w, y_2^w) \le \sqrt2 R_w$. Up to some choices which can be made arbitrarily, the above determines a map $w \mapsto y_1^w$.

    Recall Theorem~\ref{thm:Gruber} says that all but $o(n)$ many $y \in Y_n$ are nice. Our goal is to prove $v^+ = o(n)$ as well. Suppose, toward contradiction, that for some constant $\kappa > 0$, there are infinitely many $n$ for which $v^+ \ge \kappa n$. For the rest of the proof, we consider only such $n$. By the previous corollary, the number of vertices of each cell is uniformly bounded by a constant $K$. So by the pigeonhole principle, as $w$ varies over the $v^+ \ge \kappa n$ many high-degree vertices, the image of the map $w \mapsto y_1^w$ ranges over at least $\kappa n/K$ \emph{distinct} points of $Y_n$. So if $n$ is large enough, most of these (in particular at least one) are nice. Fix any such $w$ with $y_1^w$ nice. Having now fixed $w$ and the $y_i^w$, let's simplify the notation that follows by writing $y = y_1^w$ and $y' = y_2^w$. Again, $y$ and $y'$ are special among the points of $Y_n$ whose cells contain $w$ for having $\varrho_\M(y,y') \le \sqrt2 R_w$. By niceness, there are $z_1,\dots,z_6 \in Y_n$ such that $Y_n \cap B(y, 1.1\varrho_n) = \{y, z_1,\dots,z_6\}$ and all the $\varrho_\M(y,z_j)$ and $\varrho_\M(z_j,z_{j+1})$ are $(1\pm\varepsilon_n)\varrho_n$ where $\varepsilon_n,\varrho_n$ are as in Theorem~\ref{thm:Gruber} (in particular, $\varepsilon_n \to 0$).

    Observe that $R_w$ cannot be too large compared to $\varrho_n$, lest one of the $z_j$ lie in the interior of $B(w,R_w)$. On the other hand, if $R_w$ is too small relative to $\varrho_n$, then $y'$ will be so close to $y$ that niceness is violated. We'll see presently that the former forces the latter, thus obtaining a contradiction. Let $\varepsilon > 0$ small enough that $(1+\varepsilon)\sqrt{2/3} < 1$. Fix $\delta > 0$ small enough (say $\frac1{1-\delta} \le 1+\varepsilon/2$) and, subsequently, $n$ large enough that $\frac{1+\varepsilon_n}{1-\delta} \le 1+\varepsilon$. Take $\vartheta > 0$ small enough that $\cos(\frac\pi6 + \vartheta) \ge (1-\delta)\frac{\sqrt 3}{2}$. Note that $\varepsilon$, $\delta$, and $\vartheta$ do not depend on $n$ or $w$ or any other peculiarities, they are universal constants. Consider the angle $\alpha_j = \angle z_jy z_{j+1}$. This angle is as large as it can possibly be if $\varrho_\M(y,z_j) = \varrho_\M(y,z_{j+1}) = (1-\varepsilon_n)\varrho_n$ and $\varrho_\M(z_j,z_{j+1}) = (1+\varepsilon_n)\varrho_n$ in which case, $\sin\frac{\alpha_j}{2} = \frac12\frac{1+\varepsilon_n}{1-\varepsilon_n}$. Since $\varepsilon_n \to 0$, we can therefore ensure that $\alpha_j \le \alpha^* \coloneqq \frac\pi3 + 2\vartheta$ by taking $n$ large enough. 

    Summarizing so far, each cyclically consecutive pair $z_j,z_{j+1}$ are separated by an angle of no more than $\alpha^*$ with respect to $y$. This puts an upper bound on $R_w$ for if the interior of $B(w,R_w)$ contains an arc of angle $> \alpha^*$ of the circle centered at $y$ of radius $(1+\varepsilon_n)\varrho_n$, then $B(w,R_w)$ necessarily contains some $z_j$. The extremal situation, with $R_w$ taking the maximum possible value $R^*$, is illustrated in Figure~\ref{fig:high_deg_vertex}. At least one $z_j$ lies in the shaded region of the diagram (possibly on its boundary) so if $R_w$ is any larger, then this region is interior to $B(w,R_w)$. We obtain $\frac{(1+\varepsilon_n)\varrho_n}{2R^*} = \cos(\alpha^*/2) = \cos(\frac\pi6 + \vartheta) \ge (1-\delta)\frac{\sqrt{3}}{2}$ and thus $R_w \le R^* \le \frac{1+\varepsilon_n}{1-\delta} \frac{\varrho_n}{\sqrt 3} \le (1+\varepsilon)\frac{\varrho_n}{\sqrt 3}$. 
    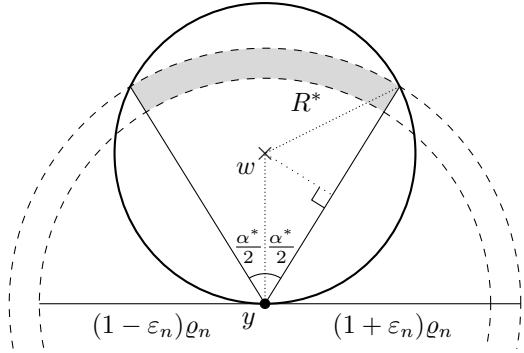
\begin{figure}[ht]
        \begin{tikzpicture}[scale=2]
            \pgfmathsetmacro\l{1.7}
            \clip (-\l-.1,-.3) rectangle (\l+.1,2.1);
            \coordinate (w) at (0,1);
            \coordinate (y) at (0,0);
            \coordinate (p+) at ({sqrt(\l^2-\l^4/4)},\l^2/2);
            \coordinate (p-) at ({-sqrt(\l^2-\l^4/4)},\l^2/2);
            \coordinate (q) at ($(y)!.5!(p+)$);
            \draw [thick] (w) circle (1);
            \fill (y) node[below left]{$y$} circle (1pt);
            \node at (w) {$\times$};
            \node[below left] at (w) {$w$};
            \draw [dashed] (y) circle (\l);
            \draw [dashed] (y) circle (\l-.2);
            \draw (p-) -- (y) -- (p+);
            \draw[densely dotted] (y) -- (w) -- node[midway,above left]{$R^*$} (p+);
            \draw (y) -- node[midway, below] {$(1+\varepsilon_n)\varrho_n$} (\l,0);
            \draw (y) -- node[midway, below] {$(1-\varepsilon_n)\varrho_n$} (-\l+.2,0);
            \draw [dotted] (w) -- (q);
            \draw pic[draw, angle radius=.4cm, angle eccentricity=2, "$\tfrac{\alpha^*}{2}$"] {angle=p+--y--w};
            \draw pic[draw, angle radius=.4cm, angle eccentricity=2, "$\tfrac{\alpha^*}{2}$"] {angle=w--y--p-};
            \draw pic[draw, angle radius=.2cm] {right angle=y--q--w};
            \pgfmathsetmacro\a{atan(\l^2/2/sqrt(\l^2-\l^4/4))}
            \fill[black, opacity=.15] ($(y)!{1-.2/\l}!(p+)$) -- (p+) arc (\a:180-\a:\l) -- ($(y)!{1-.2/\l}!(p-)$) arc (180-\a:\a:\l-.2);
        \end{tikzpicture}
        \caption{Extremal situation near nice $y \in Y_n$ with high degree vertex $w$. By niceness, some point of $Y_n$ must lie in the shaded region (possibly on its boundary). Meanwhile, since $w$ is a vertex of the cell $D_y$, no point of $Y_n$ can lie on the interior of the solid circle (center $w$, radius $R_w = R^*$). This lets us calculate the maximum possible radius $R^*$.}
        \label{fig:high_deg_vertex}
    \end{figure}

    Remember that $y'$ is less than a quarter way around $C_w$ from $y$, so $\varrho_\M(y,y') \le \sqrt 2 R_w \le (1+\varepsilon)\sqrt{\frac23}\varrho_n$. This is definitely $\le \varrho_n$, hence $y'$ is among the $z_j$. But then $\varrho_\M(y,y') \ge (1-\varepsilon_n)\varrho_n$. So
    \[
        (1-\varepsilon_n)\varrho_n
        \le \varrho_\M(y,y')
        \le (1+\varepsilon)\sqrt{\frac23}\varrho_n
    \]
    and $\varepsilon_n\ge 1 - (1+\varepsilon)\sqrt\frac23 > 0$, thereby contradicting the fact that $\varepsilon_n \to 0$. Therefore, there can be no such positive constant $\kappa$ and so $v^+ = o(n)$. Having already shown that $v = \Omega(n)$, this proves the claim that $v^+/v \to 0$.
\end{proof}

\subsection{The regular polygon moment \texorpdfstring{$\m(a,k)$}{m(a,k)}}
\label{sec:polygon_moment}
Here we introduce a key function, the regular polygon moment $\m(a,k)$, and comment on its important properties. 

For $a > 0$ and $k$ an integer $\ge 3$, let $R(a,k) \subset \R^2$ be (the interior of) a regular $k$-gon of area $a$ centered at the origin and define
\[
    \m(a,k) = \int_{R(a,k)} \norm x^r\,dx
    = \frac{k}{\frac r2 + 1}\Big(\frac{a}{k\tan\frac\pi k}\Big)^{\frac r2 + 1} \int_0^{\pi/k} \frac{1}{\cos^{r+2}\theta}\,d\theta.
\]
The explicit formula on the right is straightforward to obtain and is perfectly well defined for non-integer $k \ge 3$. Thus, we take this formula as the definition of $\m : (0,\infty) \times [3,\infty) \to \R$. The polygon moment has the following key properties we will use freely in our arguments to come.
\begin{enumerate}
    \item (Monotonicity in sections.) $\m(a,k)$ is increasing in $a$ for each fixed $k$ (obviously) and decreasing in $k$ (toward the disc limit $\int_{\norm{x}\le \sqrt{\smash[b]{a/\pi}}} \norm x^r\,dx > 0$) for each fixed $a$.
    \item (Fejes T\'oth's moment lemma.) If $P \subset \R^2$ is any convex polygonal region with $k$ sides and area $a$, then
        \[
            \int_P \norm x^r\,dx \ge \m(k,a),
        \]
        that is, the $\norm\cdot^r$ moment of any convex polygon (with respect to any origin after a translation) is no less than the moment of the regular polygon of the same area and having the same number of sides. See~\cite{Toth1973} for a proof.
    \item (Uniform decay as $a \to 0$ of least eigenvalue of the Hessian.) In \cite{Gruber1999}, Gruber proves convexity of $\mu$ by showing positive definiteness of the Hessian. We require a strengthening of this result, namely, we want an explicit lower bound on the least eigenvalue of the Hessian near $a=0$ which holds \emph{uniformly} for all $k$ in any given compact interval. We make this precise in the following lemma.
\end{enumerate}

\begin{lemma}
    Let $\lambda_0 = \lambda_0(a,k)$ be the least eigenvalue of the Hessian of $\mu$. For any $K \ge 3$, there are constants $\varepsilon_K,C_K,C_K' > 0$ such that $C_Ka^{\frac r2+1} \le \lambda_0 \le C_K'a^{\frac r2+1}$ for all $k \in [3,K]$ whenever $a < \varepsilon_K$.
    \label{lemma:eigenvalue_decay}
\end{lemma}

\begin{remark}
    The conclusion of this lemma is a strictly stronger assertion than merely $\lambda_0(a,k) = \Theta(a^{\frac r2+1})$ as $a \to 0$ for each $k \in [3,K]$ because we are further claiming that \emph{the same constants} $\varepsilon_K, C_K, C_K'$ (depending only on $K$) witness this asymptotic relation uniformly for all $k \in [3,K]$. In other words, the constants $C_K,C_K'$ do the job on the entire rectangle $(0,\varepsilon_K] \times [3,K]$. 
\end{remark}

\begin{proof}
    Following Gruber's calculation\footnote{His ``lengthy calculation'' on \cite[p.~294]{Gruber1999} has some typos but they are all easily fixed and the final line is correct, as the reader is welcome to verify. For those who wish to do so, we mention that the only step he omits which we feel is not immediately clear is, in his notation, the equality $\frac{2a\cos\frac\pi v}{v\sin\frac\pi v}\int_0^{\pi/v} g'\Big(\frac{h}{\cos^2\frac\pi v}\Big)\frac{\sin^2\psi}{\cos^4\psi}\,d\psi = \frac{\sin\frac\pi v}{\cos\frac\pi v}K - I$ but this is merely an integration by parts exercise.} from \cite{Gruber1999} in the special case of $f(t) = t^r$, one finds that the second partial derivatives of $\mu$ are of the form 
    \[
        \mu_{aa} = a^{\frac r2-1}A(k),
        \qquad \mu_{ak} = a^{\frac r2}B(k),
        \qquad \mu_{kk} = a^{\frac r2+1}C(k)
    \]
    and the determinant of Hessian is 
    \[
        \mu_{aa}\mu_{kk} - \mu_{ak}^2 = a^rD(k)
    \]
    where $A,B,C,D$ are each continuous on $[3,\infty)$ and do not depend on $a$. Furthermore, $A$ and $D$ are strictly positive. (This of course suffices to demonstrate positive definiteness.) 

    Now fix $K \ge 3$. Letting $a^* = \big(\frac{\inf_{[3,K]}A}{\sup_{[3,K]}\abs C}\big)^{1/2} > 0$ (where this is understood as $a^* = +\infty$ if $\sup_{[3,K]} \abs C = 0$), we have that $\mu_{aa} - \mu_{kk} = a^{\frac r2-1}(A - a^2C) > 0$ whenever $a < a^*$. Hereafter, assume $a < a^*$.

    Considering the characteristic polynomial of the Hessian of $\mu$, we obtain an explicit expression for $\lambda_0$ in terms of the derivatives of $\mu$ using the quadratic formula:
    \begin{align*}
        2\lambda_0
        &= (\mu_{aa} + \mu_{kk}) - \sqrt{(\mu_{aa} + \mu_{kk})^2 - 4(\mu_{aa}\mu_{kk} - \mu_{ak}^2)}
        \\&= (\mu_{aa} + \mu_{kk}) - \sqrt{(\mu_{aa} - \mu_{kk})^2 + 4\mu_{ak}^2}
        \\&= (\mu_{aa} + \mu_{kk}) - (\underbrace{\mu_{aa} - \mu_{kk}}_{>0})\sqrt{1 + \Big(\frac{2\mu_{ak}}{\mu_{aa}-\mu_{kk}}\Big)^2}
        \\&= a^{\frac r2-1}\bigg[ (A + a^2C) - (A - a^2C)\sqrt{1 + \Big(\frac{2B a}{A-a^2C}\Big)^2}\bigg]
    \end{align*}
    Define $F : (-a^*,a^*)\times[3,K] \to \R$ by $F(a,k) = \sqrt{1 + \big(\frac{2Ba}{A-a^2C}\big)^2}$. Since $A > 0$, it follows that $F(\cdot,k)$ is four times continuously differentiable (in fact, analytic) at $a = 0$ for each fixed $k$. By Taylor's theorem (for functions of one variable), for each $a \in [0,a^*/2]$,
    \[
        F(a,k) = 1 + \frac{2B^2}{A^2}a^2 + \frac{1}{4!}F_{aaaa}(\xi,k)a^4
    \]
    for some $\xi \in (0,a) \subset [0,a^*/2]$ (which may depend on $k$). Substituting this into the preceding equation and combining like terms in $a$ yields
    \[
        2\lambda_0
        = a^{\frac r2-1}\bigg[ \frac2A(AC - B^2)a^2 - \frac1{4!}(A-a^2C)F_{aaaa}(\xi,k)a^4 \bigg]
    \]
    Now, $F_{aaaa}$ is continuous (as a function of two variables), hence bounded on $[0,a^*/2]\times [3,K]$. Recall, $A$ and $C$ are also continuous. So take 
    \[
        \alpha = \frac1{4!}\sup_{(a,k) \in [0,a^*/2]\times[3,K]} \abs{A-a^2C}\cdot \sup_{(\xi,k)\in[0,a^*/2]\times[3,K]}\abs{F_{aaaa}(\xi,k)}
        < \infty
    \]
    and note that $\alpha$ depends only on $K$. 

    Observe also that $AC - B^2$ is just a rescaling of the determinant of the Hessian, namely, $AC - B^2 = a^{-r}(\mu_{aa}\mu_{kk}-\mu_{ak}^2) = D(k)$. Recall that $D$ and $A$ are both strictly positive and continuous on $[3,\infty)$ and thus, $0 < \beta \le \frac DA \le \gamma < \infty$ where $\beta = \frac{\inf_{[3,K]} D}{\sup_{[3,K]}A}$ and $\gamma = \frac{\sup_{[3,K]} D}{\inf_{[3,K]} A}$ depend only on $K$.

    In all, we have that
    \[
        a^{\frac r2-1}\big( 2\beta a^2 - \alpha a^4 \big)
        \le 2\lambda_0
        \le a^{\frac r2-1}\big( 2\gamma a^2 + \alpha a^4 \big)
    \]
    provided $a \le a^*/2$. Finally, if $a^2 \le \beta/\alpha$ (where again, this imposes no restriction at all if it should happen that $\alpha = 0$), then $\beta a^2 \le 2\beta a^2 - \alpha a^4$ and $2\gamma a^2 + \alpha a^4 \le 3\gamma a^2$ and thus $\frac12 \beta a^{\frac r2+1} \le \lambda_0 \le \frac32\gamma a^{\frac r2+1}$ whenever $a \le \min\{a^*/2, \sqrt{\beta/\alpha}\}$. 
\end{proof}

\subsection{Interpolating between partial asymptotic upper bounds}
\label{sec:interpolation_lemma}
The last tool we need before moving on to the main results is a means of asymptotically estimating a function on the natural numbers for which we only have partial information. Looking ahead a little, we will produce by constructive methods explicit asymptotic upper bounds for the minimum $\E(Y_n)$ but only on a certain infinite family of values of $n$, not on all $n$. Nonetheless, we can use the obvious fact that $\E(Y_n)$ is non-increasing as a function of $n$ to ``fill in the gaps'', as it were, by appealing to the following lemma.

To state the lemma, we must introduce the notion of a \emph{gap function}. 
For an unbounded set $S$ of natural numbers, if $s \in S$, write $s^+$ for the next highest element of $S$ above $s$, i.e., $s^+ = \min S\setminus [0,s]$. Define 
\[
    \text{$\gap_S(n) = s^+-s$ where $s \in S$ is such that $n \in [s,s^+)$}.
\]
We'll say an unbounded set $S$ satisfies the \emph{small gap condition} if $\gap_S(n) = o(n)$ as $n \to \infty$.

For an arbitrary set $S$, the gap function $\gap_S$ may be very erratic. Sometimes, it is convenient to consider instead $\maxgap_S(n) = \max_{m\le n}\gap_S(m)$. Obviously, $\maxgap_S$ is non-decreasing and $\gap_S \le \maxgap_S$. Moreover, if $\gap_S \le u$ for some non-decreasing $u$, then $\maxgap_S \le u$. 

\begin{lemma}[{$\frac1n + O(e(n))$ interpolation}]
    Let $F : \N \to \R^+$ be a non-increasing function. Suppose $S\subseteq \N$ satisfies the small gap condition and that
    \[
        F(n) = \frac1n + O(e(n))
        \quad\text{for all $n\in S$}
    \]
    for some $e : \R^+ \to \R^+$ (the ``error function'') such that
    \begin{enumerate}
        \item $e(n) = o(n^{-1})$ as $n\to\infty$,
        \item $e$ is non-increasing,
        \item\label{item:homog} $e$ has the following homogeneity property: for all $\alpha > 0$, there is $\beta > 0$ such that for all sufficiently large $n$, the inequality $e(\alpha n) \le \beta e(n)$ holds.
    \end{enumerate}
    Then,
    \[
        F(n) = \frac1n + O(\max(n^{-2}\gap_S(n), e(n)))
        \quad\text{for all $n \in \N$}.
    \]
    Moreover, this bound is sharp in the sense that possibly there is a constant $C > 0$ such that $F(n) \ge \frac1n + Cn^{-1}\gap_S(n)$ infinitely often. In other words, under these hypotheses, one cannot generally attain a $o(n^{-2}\gap_S(n))$ bound.
    \label{lemma:error_interp}
\end{lemma}
\begin{proof}
    By the small gap condition, assume $n$ is large enough that $\gap_S(n) \le n/2$. Let $s \in S$ such that $n \in [s,s^+)$. Then,
    \begin{alignat*}{2}
        F(n) &\le F(s)    &\qquad&\text{since $F$ is non-increasing},
        \\&= \frac1s + O(e(s))      &&\text{since $s \in S$},
        \\&= \frac1n + \frac{n-s}{sn} + O(e(s))
        \\&\le \frac1n  + \frac{\gap_S(n)}{(n-\gap_S(n))n} + O(e(n-\gap_S(n)))     &&\parbox{10em}{bc.\ $n < s^+ = s+\gap_S(n)$ and $e$ is non-increasing,}
        \\&\le \frac1n + 2\cdot \frac{\gap_S(n)}{n^2} + O(e(n/2))   &&\text{since $\gap_S(n) \le n/2$,}
        \\&= \frac1n + O(\max(n^{-2}\gap_S(n), e(n)))   &&\parbox{10em}{by the homogeneity property of $e$.}
    \end{alignat*}
    For the sharpness result, one need only consider the function $G$ given by $G(n) = F(s)$ for all $n \in [s,s^+)$, i.e., the extension of $F\upharpoonright S$ which is constant on each interval $[s,s^+)$. One checks that for each $n$ of the form $n = s^+-1$, one has $G(n) \ge \frac1n + \frac{\gap_S(n)}{n^2}$. Since there are infinitely many such $n$ and $G$ itself satisfies the hypotheses on $F$, this yields a sufficient example.
\end{proof}

\begin{remark}
    The lemma as stated above is more general than we will require in what follows. In particular, on the hexagonal torus, it will turn out that the error function $e$ is identically zero. Aside from the fact the general proof is elementary, we feel the more general lemma is worth stating because it is the core of the method. The problem of producing tight upper bounds on $\E(Y_n)$ \emph{for all $n$} is very difficult for myriad topological and geometric reasons. The interpolation lemma allows us to consider only ``good'' $n$ for which we can exercise tight control over the error, so long as they do not thin out too much as $n \to \infty$.
\end{remark}

\section{The hexagonal torus}
\label{sec:torus}
For any lattice $\Lambda$ in $\R^2$, the quotient $\T = \R^2/\Lambda$ is a \emph{flat torus}, a parallelogram with its opposite sides identified, as depicted below. We consider in this section the so-called \emph{hexagonal} torus where $\Lambda$ is spanned by $(1,0)$ and $(\frac12, \frac{\sqrt3}{2})$. Two ways of conceiving of this surface are shown in Figure~\ref{fig:hex_torus}. (The parallelogram is a rhombus with acute angle $\pi/3$.) The arguments that follow generalize to a boarder class of flat tori with certain aspect ratios.
\begin{figure}[ht]
    \begin{tikzpicture}[scale=1.4, thick]
        \begin{scope}[xshift=0cm]
            \foreach \i in {1,...,6}
            \coordinate (\i) at (60*\i:1);
            \fill[green!20] (6) -- (1) -- (3) -- (5) -- cycle;
            \fill[red!20] (1) -- (2) -- (3) -- cycle;
            \fill[blue!20] (5) -- (4) -- (3) -- cycle;
            \draw[thick, postaction={decorate}, >={To[scale=.9]},
                decoration={markings, mark=at position .1 with {\arrow{>}}},
                decoration={markings, mark=at position .27 with {\arrow{>>}}},
                decoration={markings, mark=at position .44 with {\arrow{>>>}}},
                decoration={markings, mark=at position .59 with {\arrow{<}}},
                decoration={markings, mark=at position .76 with {\arrow{<<}}},
                decoration={markings, mark=at position .94 with {\arrow{<<<}}},
            ]
                (1) -- (2) -- (3) -- (4) -- (5) -- (6) -- cycle;
                \draw[dashed, postaction={decorate}, decoration={markings, mark=at position .56 with {\arrow{Triangle}}}] (1) -- (3);
                \draw[dashed, postaction={decorate}, decoration={markings, mark=at position .6 with {\arrow{Triangle[]Triangle[]}}}] (5) -- (3);
        \end{scope}
        \begin{scope}[xshift=3cm]
            \foreach \i in {1,...,6}
            \coordinate (\i) at (60*\i:1);
            \fill[green!20] (6) -- (1) -- (3) -- (5) -- cycle;
            \fill[red!20] (2,0) -- (6) -- (5) -- cycle;
            \fill[blue!20] (2,0) -- (6) -- (1) -- cycle;
            \draw[thick] (1) -- (3) -- (5) -- (2,0) -- cycle;
            \draw[dashed] (1) -- (6) -- (5);
            \draw[dashed] (6) -- (2,0);
            \begin{scope}[>={Triangle}]
                \draw decorate [decoration={markings, mark=at position .56 with {\arrow{>}}}] {(1) -- (3)};
                \draw decorate [decoration={markings, mark=at position .56 with {\arrow{>}}}] {(2,0) -- (5)};
                \draw decorate [decoration={markings, mark=at position .6 with {\arrow{>>}}}] {(5) -- (3)};
                \draw decorate [decoration={markings, mark=at position .6 with {\arrow{>>}}}] {(2,0) -- (1)};
            \end{scope}
            \begin{scope}[>={To[scale=.9]}]
                \draw decorate [decoration={markings, mark=at position .6 with {\arrow{>>>}}}] {(1) -- (6)};
                \draw decorate [decoration={markings, mark=at position .56 with {\arrow{>>}}}] {(6) -- (5)};
                \draw decorate [decoration={markings, mark=at position .6 with {\arrow{>}}}] {(2,0) -- (6)};
            \end{scope}
        \end{scope}
    \end{tikzpicture}
    \caption{Two equivalent representations of the hexagonal torus. The name obviously refers to the representation on the left but the rhombus representation on the right (particularly with side length 1) is more convenient for our purposes. Moreover, our results generalize to other parallelograms akin to this rhombus as discussed at the end of the section.}
    \label{fig:hex_torus}
\end{figure}
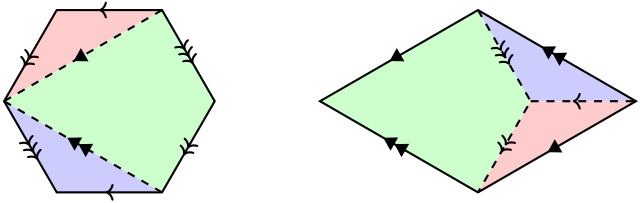

We now state the main theorem of this section and a more simply stated corollary. The proof shall be divided into several lemmas. 

\begin{theorem}
    Let $Y_n \subset \T$ with $\abs{Y_n} = n$ be a minimizer of $\E$ for each $n$. Then the number of defects in $Y_n$ is $O(\gap_\Lo(n))$ and the variance (i.e., the average squared deviation from the mean) of the areas of the Voronoi cells is $O(n^{-1}\gap_\Lo(n))$ where $\Lo = \{p^2 + pq + q^2 \mid p,q \in \N\}$ is the set of L\"oschian numbers (and $\gap_\Lo$ is as in \textsection\ref{sec:interpolation_lemma}).
    \label{thm:hex_torus}
\end{theorem}

\begin{corollary}
    The minimizer $Y_n$ has no more than $O(n^{1/4})$ many defects and the variance of the areas of the Voronoi cells is at most $O(n^{-3/4})$.
    \label{cor:hex_torus_power_law_bound}
\end{corollary}

\begin{proof}[Proof of corollary]
    Obviously, assuming Theorem~\ref{thm:hex_torus}, it suffices to show $\gap_\Lo(n) = O(n^{1/4})$. We provide an elementary proof adapted from that in \cite{Shiu2013} of the Bambah-Chowla theorem (which originally appeared in \cite{Bambah1947}). A nice generalization (which includes our case) appears in \cite{Mordell1969}. 

    First, the set of L\"oschian numbers is equivalently $\Lo = \{p^2 + 3q^2 \mid p,q \in \N\}$ (see for instance \cite{Watson1979}). Now, we'll show that for every real $t \ge 1$, there are integers $p,q$ such that $t < p^2 + 3q^2 < t + 2\sqrt6 t^{1/4} + 3$. Let $p,q$ be the integers satisfying $\sqrt t -1 < p \le \sqrt t$ and $s < q \le s+1$ where $s = \sqrt{\frac{t - p^2}{3}}$. Then
    \[
        t = p^2 + 3s^2 < p^2 + 3q^2 = t - 3s^2 + 3q^2 \le t + 6s + 3
    \]
    and $3s^2 = t - p^2 < t - (\sqrt t -1)^2 = 2\sqrt t - 1 < 2\sqrt t$ hence $s < \sqrt{\frac23}t^{1/4}$.

    Now, take any $n \ge 1$ and let $\ell < \ell^+$ be the consecutive L\"oschian numbers for which $\ell \le n < \ell^+$. By the above, taking $t = \ell$, there is some L\"oschian number $\ell'$ in the interval $(\ell, \ell + 2\sqrt 6\ell^{1/4} + 3)$ and clearly $\ell < \ell^+ \le \ell'$ (since $\ell^+$ is the least L\"oschian number above $\ell$). Thus, $\gap_\Lo(n) = \ell^+ - \ell < 2\sqrt 6 \ell^{1/4} + 3 \le 2\sqrt 6 n^{1/4} + 3$. 
\end{proof}

It is speculated that $\gap_\Lo(n)$ is in fact $O(n^\delta)$ or even $O(\log n)$. Nevertheless, an upper bound of $O(n^{1/4})$, being strictly better than $O(\sqrt n)$, provides an affirmative answer to the main question in the case of the hexagonal torus.

Before proceeding, we make some remarks on the theorem.

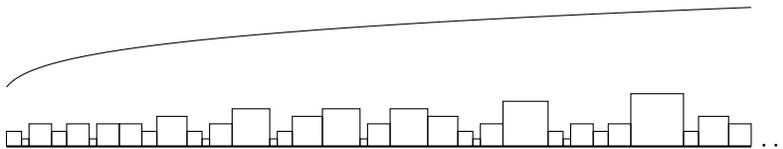
\begin{figure}[t]
    \begin{center}
        \begin{tikzpicture}[domain=1:100, scale=.1, samples=200]
            \draw[thick] plot (\x,0);
            \foreach \a/\b in {1/3, 3/4, 4/7, 7/9, 9/12, 12/13, 13/16, 16/19, 19/21, 21/25, 25/27, 27/28, 28/31, 31/36, 36/37, 37/39, 39/43, 43/48, 48/49, 49/52, 52/57, 57/61, 61/63, 63/64, 64/67, 67/73, 73/75, 75/76, 76/79, 79/81, 81/84, 84/91, 91/93, 93/97, 97/100} {
                \draw (\a,0) -- (\a,\b-\a) -- (\b,\b-\a) -- (\b,0);
            }
            \draw plot (\x,{2*sqrt(6)*pow(\x,1/4) + 3});
            \node[right] at (100,0) {$\cdots$};
        \end{tikzpicture}
    \end{center}
    \caption{Visualization of $\gap_\Lo$. The left endpoint of the horizontal axis represents $1$ and the base of each square is an interval between consecutive L\"oschian numbers. Thus, the upper edges of the squares (with their left endpoints only) form the graph of $\gap_\Lo$. The curve is the graph of of the upper bound $2\sqrt 6 n^{1/4} + 3$.}
    \label{fig:Loschian_gaps}
\end{figure}

\begin{remark}
    The function $\gap_\Lo$ (visualized in Figure~\ref{fig:Loschian_gaps}) is positive so $O(\gap_\Lo(n))$ is a sensible expression which defines a class of positive functions on $\N$. But admittedly, $\gap_\Lo$ is a rather strange function to appear in a big-O expression. In particular, it is definitely not non-decreasing. In fact, $\gap_\Lo(n) = 1$ infinitely often since $p^2 + pp + p^2 = 3p^2$ and $(p-1)^2 + (p-1)(p+1) + (p+1)^2 = 3p^2 +1$ for every $p \ge 1$ and similarly, $\gap_\Lo(n) \le k^2$ infinitely often for every $k$. If one wants a non-decreasing function, the conclusion of the theorem holds with $\gap_\Lo(n)$ replaced by $\max_{m\le n} \gap_\Lo(m)$ but we encourage the reader to take onboard the result involving the honest gap function, even if it is a somewhat unorthodox big-O assertion. Just take care when working with expressions like $O(\gap_\Lo(n))$.
\end{remark}

\begin{remark}
    The fact that $\gap_\Lo$ is unbounded follows from the well-known fact $\frac{\abs{\Lo \cap [0,n]}}{n} \sim \frac{C}{\sqrt{\log n}}$ for a constant $C$ as $n \to \infty$, i.e., the asymptotic density of $\Lo$ is like $C/\sqrt{\log n}$, which in turn is a special case of a classical result of Bernays in his 1912 dissertation. This is the sense in which the average gap between L\"oschian numbers below $n$ is proportional to $\sqrt{\log n}$. \textit{Nota bene}, the average gap size below $n$ is not to be confused with the average over all $m \le n$ of the gap containing $m$, for the larger gaps of course contain more of these $m$ than each ``typical'' gap. This distinction amounts to the average of the side lengths of the squares in the diagram in Figure~\ref{fig:Loschian_gaps} versus their average area. 
\end{remark}

The structure of the proof is as follows. As in the theorem statement, let $Y_n$ denote a minimizer of $\E$ having exactly $n$ points. 
\begin{enumerate}
    \item First, we provide an asymptotic lower bound on $\E(Y_n)$ whose leading order term involves regular hexagons and whose next highest order term is related to the number of defects. The key tools here are the properties of $\mu(a,k)$, detailed in \S\ref{sec:polygon_moment}.
    \item We then give an asymptotic upper bound for $\E(Y_n)$ along a certain \emph{subsequence} (related to $\Lo$) which coincides with the previously obtained lower bound to first order, thus squeezing the second order term.
    \item Finally, using the fact that $\E(Y_n)$ is non-increasing in $n$, we apply the interpolation lemma to ``fill in the gaps'' between values for which we already have an upper bound.
\end{enumerate}

\begin{lemma}[Lower bound]
    There is a constant $\alpha>0$ such that for every $\lambda < 1$, if $n$ is sufficiently large, 
    \[
        \E(Y_n) \ge \frac{1}{n^{\frac r2}} \m(\sigma(\T), 6) + \frac{\alpha}{n^{\frac r2+1}}\Big( n\s^2(Y_n) + \lambda \df(Y_n) \Big)
    \]
    where $\m$ is as in \S\ref{sec:polygon_moment} and $\s^2(Y_n) = \frac1n\sum_y \abs{\sigma(D_y) - \sigma(\T)/n}^2$ with $D_y$ the Voronoi cell of $y$ and $\df(Y_n)$ is the number of defects. (The notation $\s^2$ is intentionally suggestive of population variance, although the data $\sigma(D_y)$ are of course not random.)
    \label{lemma:hex_lower_bound}
\end{lemma}

Before we give the proof, this lemma warrants some commentary.

\begin{remark}
    The second term on the right of the above inequality is not self-evidently of lower order than the first term but this is indeed so as $\E(Y_n) = n^{-r/2}\mu(\sigma(\T),6) \pm o(n^{-r/2})$ (see \cite{Gruber2001}). The inequality is written in a way that will be convenient later when comparing with the upper bound.
\end{remark}

\begin{remark}
    This lower bound, as stated, is particular to the flat case where we can apply the moment lemma directly. However, the same lower bound holds to within $\pm O(n^{-\frac r2-1})$ in many other cases. For example, on the 2-sphere, one can employ gnomonic projections to reduce to the Euclidean setting while controlling the error. Constant negative curvature surfaces admit analogous projections. In the non-constant curvature setting, one cannot hope for such projections which take Voronoi cells to Euclidean polygons, but one still has natural projections where approximate polygons will do.
\end{remark}

\begin{remark}
    Note that while this provides a lower bound on minimum energy, i.e., the energy of the optimal point configuration $Y_n$, it does not necessarily hold for arbitrary point configurations because the bound itself depends on the point configuration. Indeed, one can easily construct severe examples $X_n$ having, say, only quadrilateral cells, so that $\df(X_n) = n$. In this case, the analogous lower bound is merely $\E(X_n) = \Omega(n^{-r/2})$.
\end{remark}

\begin{proof}
    For each $y \in Y_n$, let $D_y$ be the Voronoi cell of $y$ with respect to $Y_n$. Recall, $D_y$ is always a convex polygon. Let $a_y = \sigma(D_y)$ be the area of $D_y$ and let $k_y$ denote the number of edges of $D_y$. Let $\bar a = \frac1n\sum_y a_y$ and $\bar k = \frac1n \sum_y k_y$ denote their averages. By Corollary~\ref{cor:Delone}, the diameter of $D_y$ is $O(1/\sqrt n)$. In particular, the periodic structure of the torus has no bearing on distances between points within any one cell and so $D_y$ can be regarded as a Euclidean planar region by considering a chart that contains it. In local coordinates covering any one cell, we may replace $\varrho_\T(x,y)$ by $\norm{x-y}$ (the usual Euclidean 2-norm) and the surface measure $d\sigma(x)$ by the Lebesgue measure $dx$ without issue, and we shall do so without further comment. 

    By the moment lemma,
    \[
        \E(Y_n) 
        = \sum_{y\in Y_n} \int_{D_y} \norm{x-y}^r\,dx
        \ge \sum_y \m(a_y, k_y).
    \]
    Taking the first order Taylor expansion of each term about $(\bar a, \bar k)$ with the Lagrange form of the remainder we obtain
    \begin{equation}
        \sum_y \m(a_y, k_y)
        \ge n\m(\bar a, \bar k) + \frac12\inf_{(a,k)\in R_n}\lambda_0(a,k) \sum_y \norm{(a_y,k_y) - (\bar a, \bar k)}^2
        \label{eq:hex_conv_bd}
    \end{equation}
    where $\lambda_0(a,k)$ is the least eigenvalue of the Hessian of $\m$ at $(a,k)$ and $R_n \subset \R^2$ is smallest rectangle that contains all the $(a_y,k_y)$.

    Let's address the first term on the right of inequality \eqref{eq:hex_conv_bd}. Since the $D_y$ partition $\T$, we have $\bar a = \frac1n\sum_y a_y = \sigma(\T)/n$. Since the torus has Euler characteristic $0$, Lemma~\ref{lemma:hexagon_average} yields $\bar k \le 6$. Since $\m(a,k)$ is decreasing in $k$ for fixed $a$,
    \begin{equation}
        n\m(\bar a,\bar k)
        = n\m\Big(\frac{\sigma(\T)}{n},\bar k\Big)
        = n^{-r/2} \m(\sigma(\T), \bar k)
        \ge n^{-r/2} \m(\sigma(\T), 6).
        \label{eq:hex_first_order_lower_bd}
    \end{equation}

    We now turn our attention to the second term of \eqref{eq:hex_conv_bd}. Fix $\theta \in (0,1)$ as small as desired. By Lemma~\ref{lemma:hexagon_average}, we may assume $n$ is large enough that $\abs{\bar k - 6} \le \theta$ whence
    \begin{align}
        \label{eq:hex_2nd_order_sum_bd}
        \sum_y \norm{(a_y,k_y) - (\bar a, \bar k)}^2
        &= \sum_y \abs{a_y - \bar a}^2 + \sum_y \abs{k_y-\bar k}^2
        \\\nonumber&\ge \sum_y \abs{a_y - \bar a}^2 + \sum_{y\in Y_n,\, k_y \ne 6}\abs{k_y-\bar k}^2
        \\\nonumber&\ge \sum_y \abs{a_y - \bar a}^2 + \sum_{y\in Y_n,\, k_y \ne 6}(\abs{k_y-6} - \abs{6-\bar k})^2
        \\\nonumber&\ge n\s^2(Y_n) + (1-\theta)^2\df(Y_n)
    \end{align}
    where $\df(Y_n)$ is the number of defects (non-hexagonal cells).

    Lastly, let's consider the eigenvalue. Taking constants $A_\pm$ and $K$ as in Corollary~\ref{cor:Delone} so that $\frac{A_-}{n} \le a_y \le \frac{A_+}{n}$ and $k_y \le K$ for all $y \in Y_n$, we obtain $R_n \subseteq [\frac{A_-}{n}, \frac{A_+}{n}] \times [3,K]$. Then, taking $\varepsilon_K,C_K > 0$ as in Lemma~\ref{lemma:eigenvalue_decay}, we may assume $n$ is large enough that $A_+/n < \varepsilon_K$ and so
    \begin{equation}
        \inf_{R_n} \lambda_0
        \ge \inf_{[\frac{A_-}{n},\frac{A_+}{n}]\times[3,K]} \lambda_0
        \ge \inf_{[\frac{A_-}{n},\frac{A_+}{n}]\times[3,K]} C_Ka^{\frac r2+1}
        \ge \frac{C_KA_-^{\frac r2+1}}{n^{\frac r2+1}}.
        \label{eq:eigenvalue_lower_bd}
    \end{equation}

    Substituting inequalities \eqref{eq:hex_first_order_lower_bd}, \eqref{eq:hex_2nd_order_sum_bd}, \eqref{eq:eigenvalue_lower_bd} into \eqref{eq:hex_conv_bd} yields the result (with constants $\alpha = C_KA_-^{r/2+1}$ and $\lambda = (1-\theta)^2$).
\end{proof}

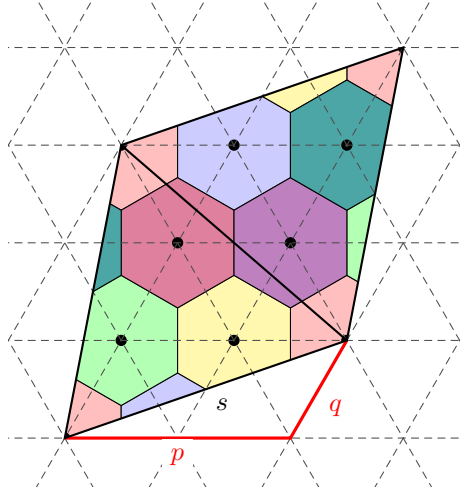
\begin{figure}[ht]
    \centering
    \begin{tikzpicture}[scale=1.5]
        \pgfmathsetmacro\r{sqrt(3)}

        \coordinate (O) at ($(1,0) + (60:1)$);

        \clip (1,\r/4) rectangle (5,5*\r/2 + \r/4);

        \begin{scope}
            \clip (O) -- ++($(0:2) + (60:1)$) -- ++($(60:2)+(120:1)$) -- ++($(180:2)+(240:1)$) -- cycle;
            \foreach \x/\y/\hue in {
                0/0/pink, .5/3/pink, 2.5/1/pink, 3/4/pink, 
                1/0/blue!20, 1.5/3/blue!20,
                .5/1/green!30, 3/2/green!30,
                1.5/1/yellow!40, 2/4/yellow!40,
                1/2/purple!50, 
                2/2/violet!50,
                2.5/3/teal!70, 0/2/teal!70
            } {
                \draw[fill=\hue] ($(O) + (\x,\y*\r/2) + (1/2, \r/6)$) foreach \i in {0,...,4} {-- ++(150 + 60*\i:\r/3)} -- cycle;
                \fill[black] ($(O) + (\x,\y*\r/2)$) circle  (.05);
            }
        \end{scope}

        \begin{scope}[densely dashed, black!70]
            \foreach \x in {-1,0,...,6} {
                \draw (\x,0) -- +(60:6);
                \draw (\x+1,0) -- +(120:6);
                \draw (0,\x*\r/2) -- +(7,0);
            }
        \end{scope}

        \draw[very thick, red] (O) -- node[midway, below, fill=white] {$p$} ++(0:2) -- node[midway, below right, fill=white] {$q$} ++(60:1);

        \draw[thick] (O) -- node[midway, below right] {$s$} ++($(0:2) + (60:1)$) -- ++($(60:2)+(120:1)$) -- ++($(180:2)+(240:1)$) -- cycle;
        \draw[thick] ($(O) + (0:2) + (60:1)$) -- +($(120:2) + (180:1)$);

    \end{tikzpicture}
    \caption{The Goldberg-Coxeter construction for $(p,q) = (2,1)$ yields the point configuration $Y_{2,1} \subset \T$ with $n = \abs{Y_{2,1}} = 7$ whose Voronoi tessellation is the well-known tiling of $\T$ by seven congruent regular hexagons. (In the proof, we scale by $1/s$ so that the rhombus is congruent for all $p,q$.) The area of each of the two large triangles (with solid edges) is clearly $s^2$ times that of each small dashed triangle, hence the number of hexagons is $2s^2 \cdot 3 / 6 = s^2 = p^2 + pq + q^2$ (the total number of small triangles within the rhombus times the number of vertices per triangle divided by six since each hexagon shares the vertices of six triangles).} 
    \label{fig:GB_example}
\end{figure}

\begin{lemma}[Upper bound via the Goldberg-Coxeter construction]
    For each $p,q \in \N$ (not both $0$), there is a point configuration $Y_{p,q} \subset \T$ with
    \[
        \E(Y_{p,q}) = \frac1{n^{r/2}} \m(\sigma(\T), 6).
    \]
    where $n = \abs{Y_{p,q}} = p^2 + pq + q^2$. Hence $\E(Y_n) \le \E(Y_{p,q})$ for all such $n$ where $Y_n$ is the minimizer.
    \label{lemma:hex_upper_bound}
\end{lemma}

\begin{proof}
    The following construction is illustrated for $(p,q) = (2,1)$ in Figure~\ref{fig:GB_example}. Let $\Lambda \subset \R^2$ be the lattice spanned by $\vec u = (1,0)$ and $\vec v = (\frac12,\frac{\sqrt3}{2})$. For $p,q \in \N$ not both zero, let $s^2 = p^2 + pq + q^2$ and let $\Delta_{p,q}$ be the closed equilateral triangular region (of side length 1) in the plane having one vertex at the origin and the other at $\frac{1}{s}(p\vec u + q\vec v) \in \frac1s\Lambda$. (Of course, there are two such triangles but it does not matter which we take.) Divide $\T$ into two equilateral triangles, say $T_1$ and $T_2$, of side length 1 in the obvious way and assign them consistent orientations (since $\T$ itself is orientable). Identify $T_i$ with $\Delta_{p,q}$, call the identification $f_{T_i} : \Delta_{p,q} \to T_i$, such that $f_i$ preserves orientation. In other words, the identifications $f_{T_i}$ are such that the orientations of the images are consistent. Finally, set $Y_{p,q} = \bigcup_i f_{T_i}(\frac{1}{s}\Lambda \cap \Delta_{p,q})$. 

    Since $\frac1s\Lambda$ is invariant under the rotations of $\Delta_{p,q}$, it follows that $Y_{p,q}$ does not depend on which particular isometries $f_{T_i}$ are used in the construction, since they are related to each other by a rotation of $\Delta_{p,q}$. (Notice also that the reflection of $Y_{p,q}$ is none other than $Y_{q,p}$ so it makes no difference which orientation we pick for $\T$ or which of the two possible triangles in the plane we took $\Delta_{p,q}$ to be, the same set of patterns $Y_{p,q}$ are witnessed as we vary $p,q \in \N$.)

    The Voronoi tessellation of $Y_{p,q}$ consists of $n = \abs{Y_{p,q}}$ congruent regular hexagons. (One way to see that they are all congruent regular hexagons is to note the translation invariance of the construction. One can cut and glue the rhombus representing the torus so as to view any desired point of $Y_{p,q}$ as away from the ``boundary''.) It remains to show that $n = p^2 + pq + q^2$. Consider the triangulation dual to the Voronoi tessellation. This is simply the image of the triangulation of the plane with vertex set $\frac1s\Lambda$ by small equilateral triangles of side length $1/s$. Since $\Delta_{p,q}$ is an equilateral triangle of side length $1$, it contains exactly $s^2$ small triangles in total (whether or not some are only partly inside $\Delta_{p,q}$). So the image in $\T = T_1 \cup T_2$ contains exactly $2s^2$ small triangles whose vertices form $Y_{p,q}$. Each such small triangle has exactly three vertices and each vertex is shared by exactly six small triangles so $\abs{Y_{p,q}} = \frac16 \cdot 3 \cdot 2s^2 = s^2 = p^2 + pq + q^2$. 
\end{proof}

We can now prove Theorem~\ref{thm:hex_torus}.

\begin{proof}[Proof of Theorem~\ref{thm:hex_torus}]
    Let $\Lo = \{p^2+pq+q^2 \mid p,q\in \N\}$ be the set of L\"oschian numbers, as in the theorem statement, and let $S = \{p^2 \mid p\in \N\}$ be the set of squares. Obviously, $S \subseteq \Lo$, so $\gap_\Lo \le \gap_S$ (as functions on $\N$). Thus,
    \[
        \gap_\Lo(n)
        \le \gap_S(n) 
        = (\lfloor\sqrt n\rfloor + 1)^2 - \lfloor\sqrt n\rfloor^2
        = 2\lfloor\sqrt n\rfloor + 1
        = o(n)
    \]
    so $\Lo$ satisfies the small gap condition. Therefore, applying Lemma~\ref{lemma:error_interp} with $F(n) = \frac{n^{\frac r2-1}}{\m(\sigma(\T),6)}\E(Y_n)$ and $e = 0$ to interpolate the partial upper bound obtained in Lemma~\ref{lemma:hex_upper_bound} yields
    \[
        \E(Y_n) \le \frac1{n^{r/2}} \m(\sigma(\T),6) + O\Big(\frac1{n^{\frac r2+1}}\gap_\Lo(n)\Big).
    \]
    In light of the lower bound from Lemma~\ref{lemma:hex_lower_bound} taking $\lambda = \frac12$ say, it must be that $n\s^2(Y_n) + \frac12\df(Y_n) = O(\gap_\Lo(n))$. Both these terms are non-negative, so the bound holds for each term individually, i.e., $\df(Y_n) = O(\gap_\Lo(n))$ and $\s^2(Y_n) = O(n^{-1}\gap_\Lo(n))$. 
\end{proof}

An important point here is the crucial role played by the \emph{binary} quadratic form $p^2 + pq + q^2$. In particular, the authors first considered only the more obvious tessellations $Y_{p,0}$ arising in the $q = 0$ case. These tessellations are ``aligned'', so to speak, with the torus itself and it is quite transparent that one can partition $\T$ into $p^2$ many congruent regular hexagons (up to a constant factor). The problem is that the values $n = p^2$ are too sparse in $\N$, they thin out too rapidly. Namely, $\gap_{\{p^2 \mid p \in \N\}}(n) = \Theta(\sqrt n)$ and the interpolation lemma only gives $\E(Y_n) \le \frac1{n^{r/2}}\m(\sigma(\T),6) + O(n^{-\frac{r+1}{2}})$ from which we can conclude only that $\df(Y_n) = O(\sqrt n)$, a non-trivial result, but not a strong enough one to answer our main question. Recall, the whole point is to demonstrate that surfaces without boundary can achieve \emph{strictly fewer} than $O(\sqrt n)$ many defects.

Generalizing the above reasoning, the key observation is that the Goldberg-Coxeter construction can be carried out on any orientable surface without boundary that can be obtained by gluing congruent equilateral triangles together along their edges. Whether or not the lower bound holds on every such surface (to within a tolerable error) is a delicate matter in general. However, our arguments work just as well in the special case of flat tori $\T_\gamma = S^1 \times \gamma\frac{\sqrt 3}{2}S^1$ where $\gamma$ is a positive rational number. The lower bound argument we gave for the hexagonal torus $\T = \T_1$ applies verbatim to $\T_\gamma$. The upper bound is essentially the same except that we now have $2ab$ big triangles that make up $\T_\gamma$ where $a,b$ are integers with $\gamma =a/b$. Following the same reasoning as before, the upper bound is obtained for all $n \in ab\Lo$. Thus, the analogue of Theorem~\ref{thm:hex_torus} holds with $\gap_\Lo(n)$ replaced by $\gap_{ab\Lo}(n)$. For any unbounded $S \subseteq \N$ and constant $k > 0$, we can express $\gap_{kS}$ in terms of $\gap_S$ as $\gap_{kS}(n) = k\cdot \gap_S(n/k)$. It follows that the $O(n^{-1/4})$ corollary holds for every $\T_\gamma$ and likewise, if $\gap_\Lo(n) = O(\log n)$, then so is the number of defects in the optimizer of $\T_\gamma$. 

\section{Stability}
\label{sec:stability}

Here, we demonstrate a certain fairly weak stability result. We do not mount a detailed investigation, but merely reinforce the main result Theorem~\ref{thm:hex_torus} by showing that it holds in some neighborhood of the minimizer $Y_n$. The worry is that $\df(Y_n)$, being integer-valued, is not continuous with respect to perturbations of $Y_n$, and so might conceivably skyrocket even for arbitrarily small perturbations. In other words, it might happen that $O(\gap_\Lo(n))$ many defects is only attained on some pathological (e.g., measure zero) set which happens to include the minimizers $Y_n$. The purpose of this section is to rule out this possibility by showing that all configurations $\tilde Y_n$ in a sufficiently small open neighborhood of $Y_n$ also attain $O(\gap_\Lo(n))$ many defects.

For each $n$, let $\delta_n : Y_n \to \M$ and write $\abs{\delta_n} = \max_{y \in Y_n} \varrho_\M(y,\delta_n(y))$. Think $\delta_nY_n$ as a perturbation of $Y_n$ by no more than $\abs{\delta_n}$. 

\begin{theorem}
    Let $Y_n \subset \T$ with $\abs{Y_n} = n$ be a minimizer of $\E$ for each $n$ and let $\delta_n : Y_n \to \T$ be perturbations as above. The following hold.
    \begin{enumerate}
        \item For any closed manifold $\M$ and for every $r > 0$, if $\abs{\delta_n} = o(n^{-1/2})$, then $\delta_nY_n$ satisfy the $n^{-1/2}$-Delone properties and $\E(\delta_nY_n) \sim \E(Y_n)$.
        \item In the special case of the hexagonal torus $\M = \T$, If $r \ge 1$ and $\abs{\delta_n} = O(n^{-3/2})$ or $r < 1$ and $\abs{\delta_n} = O(n^{-\frac12-\frac1r})$, then the number of defects in $\delta_nY_n$ is $O(\gap_\Lo(n))$.
    \end{enumerate}
    \label{thm:stability}
\end{theorem}

\begin{proof}
    We argue by showing that the argument for Theorem~\ref{thm:hex_torus} generalizes to the perturbed configurations.
    The $n^{-1/2}$-Delone properties are easy. Say $C,C'>0$ as in Theorem~\ref{thm:Gruber} witness the Delone properties for $Y_n$. Since $\abs{\delta_n} = o(n^{-1/2})$, assume $n$ is large enough that $\abs{\delta_n} \le \min\{C/3\sqrt n, C'/\sqrt n\}$. Then if $y,y' \in Y_n$ are distinct, 
    \[
        \varrho_\M(\delta_ny,\delta_ny') 
        \ge \varrho_\M(y,y') - \varrho_\M(y,\delta_ny) - \varrho_\M(y',\delta_ny') 
        \ge \frac{C}{\sqrt n} - 2\abs{\delta_n}
        = C/3\sqrt n.
    \]
    So the sequence $\delta_n Y_n$ is $\Omega(n^{-1/2})$ uniformly discrete. Take $x \in \M$, then 
    \[
        \min_{\tilde y\in \delta_n Y_n}\varrho_\M(x,\tilde y) 
        = \min_{y \in Y_n} \varrho_\M(x,\delta_n y) 
        \le \min_{y \in Y_n} \varrho_M(x,y) + \abs{\delta_n}
        \le 2C'/\sqrt n.
    \]
    Therefore the sequence $\delta_nY_n$ has the $O(n^{-1/2})$ covering radius property.

    That $\E(\delta_nY_n) \sim \E(Y_n)$ is more subtle. First consider the $r \ge 1$ case.
    \begin{align*}
        0
        &\le \E(\delta_nY_n) - \E(Y_n)
        \\&= \int_\M \big(\min_{y'} \varrho_\M(x,\delta_ny')^r - \min_y \varrho_\M(x,y)^r\big)\,d\sigma(x)
        \\&= \sum_y \int_{D_y} \big( \min_{y'} \varrho_\M(x,\delta_ny')^r - \varrho_\M(x,y)^r \big)\,d\sigma(x)
        \\&\le \sum_y \int_{D_y} \big( \varrho_\M(x,\delta_ny)^r - \varrho_\M(x,y)^r \big)\,d\sigma(x)
        \\&= \sum_y \int_{D_y} r\xi_y^{r-1}(\varrho_\M(x,\delta_ny) - \varrho_\M(x,y))\,d\sigma(x)
    \end{align*}
    where the minima and sums in the above expressions are taken over $y,y' \in Y_n$ and, employing the mean value theorem, each $\xi_y$ is a number strictly between $\varrho_\M(x,y)$ and $\varrho_\M(x,\delta_ny)$ (unless these happen to be the same value, in which case the whole term vanishes anyway).
    
    Since $x \in D_y$, we find $0 \le \xi_y \le \max\{\varrho_\M(x,y), \varrho_\M(x,\delta_ny)\} \le \varrho_\M(x,y)+\varrho_\M(x,\delta_ny) \le 2\varrho_\M(x,y) + \varrho_\M(y,\delta_ny) \le O(n^{-1/2}) + \abs{\delta_n} = O(n^{-1/2})$ (because $Y_n$ has the $O(n^{-1/2})$ covering radius property). Hence, if $r \ge 1$, then $\xi_y^{r-1} = O(n^{-\frac{r-1}{2}})$. The second factor in the integrand might have either sign but clearly $\abs{\varrho_\M(x,\delta_ny) - \varrho_\M(x,y)} \le \varrho_\M(y,\delta_ny) \le \abs{\delta_n}$ so
    \begin{multline}
        \abs{\E(\delta_nY_n) - \E(Y_n)}
        \le \sum_y \int_{D_y} O(n^{-\frac{r-1}{2}})\cdot \abs{\delta_n}\,d\sigma
        \\= O(n^{-\frac{r-1}{2}}) \cdot \sum_y \sigma(D_y)\cdot \abs{\delta_n}
        = O(n^{-\frac{r-1}{2}}) \sigma(\M) \abs{\delta_n}
        = O(n^{-\frac{r-1}{2}}) \abs{\delta_n}.
        \label{eq:stab_r>=1}
    \end{multline}

    Now consider $r < 1$. For each $n$ and $y \in Y_n$, let $E_y = B(y,\abs{\delta_n}) \cup B(\delta_n y,\abs{\delta_n})$. Then, excising each $E_y$, we obtain
    \begin{align*}
        0 &\le \E(\delta_nY_n) - \E(Y_n)
        \\&\le\sum_y \int_{D_y\setminus E_y} (\varrho_\M(x,\delta_ny)^r - \varrho_\M(x,y)^r)\,d\sigma(x) 
        \\&\qquad\qquad + \sum_y \int_{D_y\cap E_y} (\varrho_\M(x,\delta_ny)^r - \varrho_\M(x,y)^r)\,d\sigma(x)
        \\&\le \sum_y \int_{D_y\setminus E_y} r\xi_y^{r-1}(\varrho_\M(x,\delta_ny) - \varrho_\M(x,y))\,d\sigma(x) 
        \\&\qquad\qquad + \sum_y \int_{D_y\cap E_y} (\varrho_\M(x,\delta_ny)^r + \varrho_\M(x,y)^r)\,d\sigma(x)
    \end{align*}
    Since $x \in D_y$, we have $\varrho_\M(x,y) = O(n^{-1/2})$ and $\varrho_\M(x,\delta_ny) \le \varrho_\M(x,y) + \abs{\delta_n} = O(n^{-1/2})$, so the integrand in the second term is $O(n^{-r/2})$. Since $\M$ is compact, $\sigma(E_y) = O(\abs{\delta_n}^2)$. Thus, the whole second term of the above is $\sum_y O(n^{-r/2})O(\abs{\delta_n}^2) = O(n^{-r/2+1}\abs{\delta_n}^2)$. 

    For the first term, since $x \not\in E_y$, we have $\xi_y \ge \min\{\varrho_\M(x,y),\varrho_\M(x,\delta_ny)\} \ge \abs{\delta_n}$. Then, since $r < 1$, we have
    \begin{multline*}
        \sum_y\int_{D_y\setminus E_y} r\xi_y^{r-1}\abs{\varrho_\M(x,\delta_ny) - \varrho_\M(x,y)}\,d\sigma(x)
        \le \sum_y \int_{D_y\setminus E_y} r\abs{\delta_n}^{r-1}\abs{\delta_n}\,d\sigma(x)
        \\\le O(\abs{\delta_n}^r) \cdot \sum_n \sigma(D_y)
        = O(\abs{\delta_n}^r) \cdot \sigma(\M)
        = O(\abs{\delta_n}^r).
    \end{multline*}
    Altogether,
    \begin{equation}
        \abs{\E(\delta_nY_n) - \E(Y_n)}
        = O(\abs{\delta_n}^r) + O(n^{-\frac r2 + 1}\abs{\delta_n}^2)
        \label{eq:stab_r<1}
    \end{equation}
    for $r < 1$. 

    Inspecting \eqref{eq:stab_r>=1} and \eqref{eq:stab_r<1} we see that for all $r > 0$, if $\abs{\delta_n} = o(n^{-1/2})$, then $\abs{\E(\delta_nY_n) - \E(Y_n)} = o(n^{-r/2})$. Since $\E(n) = \Theta(n^{-r/2})$, this gives $\E(\delta_nY_n) \sim \E(Y_n)$ for all such perturbations $\delta_n$. This completes the first part of the theorem.

    Now consider specifically the hexagonal torus $\M = \T$. Notice that Corollary~\ref{cor:Delone} has nothing to do with the optima $Y_n$ in particular, rather it holds of any point configurations $\tilde Y_n$ with the $n^{-1/2}$-Delone property. Thus, it holds equally for the perturbations $\delta_nY_n$ with $\abs{\delta_n} = o(1/\sqrt n)$.
    Moreover, since $\E(\delta_nY_n) \sim \E(Y_n)$, it follows that $\delta_n Y_n$ has the asymptotically regular hexagonal pattern property of Theorem~\ref{thm:Gruber}. Together, this means Lemma~\ref{lemma:hexagon_average} holds for $\delta_nY_n$ and Lemma~\ref{lemma:hex_lower_bound} holds for all $\delta_nY_n$ with $\abs{\delta_n} = o(n^{-1/2})$. 

    For $r \ge 1$ and $\abs{\delta_n} = O(n^{-3/2})$, Equation~\eqref{eq:stab_r>=1}, Lemma~\ref{lemma:hex_upper_bound}, and Lemma~\ref{lemma:error_interp} yield
    \[
        \E(\delta_nY_n) \le \E(Y_n) + O(n^{-r/2-1})
        \le \frac1{n^{r/2}} \m(\sigma(\T),6) + O(n^{-r/2-1}\gap_\Lo(n)).
    \]
    Since Lemma~\ref{lemma:hex_lower_bound} holds for $\delta_nY_n$, we have $\df(\delta_nY_n) = O(\gap_\Lo(n))$, as we had for $Y_n$.

    For $r < 1$ and $\abs{\delta_n} = O(n^{-\frac12-\frac1r})$, Equation~\eqref{eq:stab_r<1} yields, $\E(\delta_nY_n) \le \E(Y_n) + O(n^{-\frac r2 - 1}) + O(n^{-\frac r2 - \frac2r}) = \E(Y_n) + O(n^{-\frac r2 - 1})$ and the conclusion follows as in the $r \ge 1$ case.
\end{proof}

\begin{remark}
    Let us mention that the argument just given is quite robust. In particular, it has nothing really to do with the L\"oschian gaps $\gap_\Lo(n)$ specifically, rather it bounds $\df(\delta_n Y_n)$ in terms of the second order contribution to $\E(Y_n)$, 
    \[
        \df(\delta_nY_n) \le O\Big( n\big( n^{r/2}\E(Y_n) - \mu(\sigma(\T),6)\big) \Big),
    \]
    and so any future improvement of the second order bound on $\E(Y_n)$ will carry over to the $\abs{\delta_n} = O(n^{-\frac12 - \max\{1,\frac1r\}})$ perturbations. 
\end{remark}

This result is somewhat unsatisfactory in that $O(n^{-3/2})$ (and $O(n^{-\frac12-\frac1r})$ in the $r < 1$ case) is much smaller than the $\Theta(n^{-1/2})$ distance between nearby points of $Y_n$, so we are here considering only very small perturbations. Nonetheless, the result suffices for the non-pathology claim we set out to show, if crudely.

We find it curious that not until this section has the $r < 1$ situation behaved any worse than the intuitively much nicer $r \ge 1$ case. Whether the $O(n^{-3/2})$ for $r \ge 1$ or the $O(n^{-\frac r2 - \frac1r})$ for $r < 1$ perturbation bounds can be improved or shown to be optimal is unclear. We will not endeavour to do so here but we wish to comment on some of the issues involved, at least in the $r \ge 1$ case. Maybe one has the thought that $\abs{\delta_n} = o(n^{-1/2})$ should suffice to maintain the asymptotic bound on the number of defects, since $n^{-1/2}$ is roughly the separation between nearby points of the optima $Y_n$. However, it seems consistent with what we know that, infinitely often, the minima $Y_n$, while having very few defects, might have relatively many \emph{near} defects. More precisely, say a Voronoi cell $D_y$ is an $\varepsilon$-near pentagon if it is a hexagon but exactly one of its sides has length $\le \varepsilon$. Sweeping many details under the rug, for any sequence $\varepsilon_n = o(n^{-1/2})$, one can use the previous theorem and some Voronoi tessellation witchcraft to bound the number of $\varepsilon_n$-near pentagons in $Y_n$. The basic idea is that there is a sequence of perturbations $\delta_n$ whose size $\abs{\delta_n}$ is related directly to $\varepsilon_n$ which, in effect, collapses the $\varepsilon_n$ short edge of a non-negligible proportion of the $\varepsilon_n$-near pentagons, turning them into genuine pentagons, while maintaining the number of sides of all other cells. Now, the resulting configurations will be highly atypical, having many high degree vertices, but the point is merely to hint at how it might be the case that a relatively small perturbation ($o(n^{-1/2})$ but not $O(n^{-3/2})$) of the optima could significantly increase the number of defects. After all, our results on defects place no restrictions on the number of near defects and Gruber's asymptotically hexagonal pattern property only bounds them by $o(n)$. The other trouble when it comes to reasoning about whether such irregular optima $Y_n$ actually obtain is that the only optima we actually have systematic access to are those for $n \in \Lo$, which are perfectly regular, and indeed should be robust to $o(n^{-1/2})$ perturbations.

\section{Adapting to the sphere}
\label{sec:sphere}

\begin{figure}[th]
    \centering
    \includegraphics[width=.5\textwidth]{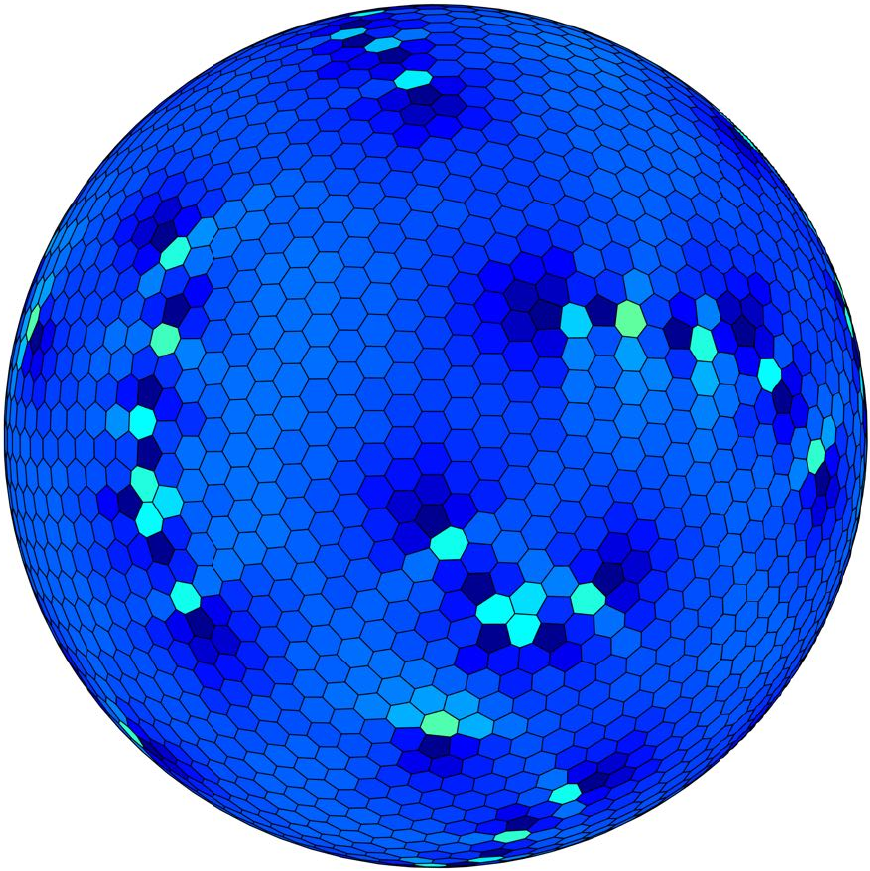}
    \caption{Example of low energy configuration on the sphere highlighting defects for $n=2500$ and $r=2$ obtained by state-of-the-art numerical methods. This figure is reprinted from \cite{Gonzalez2021a} with permission from the author.}
    \label{fig:sphereMACN}
\end{figure}
Our present work concerns the adaptation of these methods to the 2-sphere. Figure~\ref{fig:sphereMACN} shows an example of a numerically obtained low energy configuration on the sphere for $n=2500$ and $r=2$. The toolkit of Section~\ref{sec:toolkit} applies equally well to the spherical case (with suitable modifications to the proof of Lemma~\ref{lemma:hexagon_average}). A lower bound on the minimum energy can be obtain analogously to Lemma~\ref{lemma:hex_lower_bound} by passing from each Voronoi cell $D_y$ to its gnomonic projection centered at $y$. The Voronoi cells are (small) convex spherical polygons so their gnomonic projections are convex Euclidean polygons and the moment lemma applies. The projections collectively incur an error within $O(n^{-2})$, which can be tolerated. 

\begin{lemma}
    Let $Y_n \subset S^2$ be a minimizer of $\E$ with $\abs{Y_n} = n$ for each $n$. Then there is a constant $\alpha$ such that for every $\lambda < 1$,
    \[
        \E(Y_n) \ge \frac1{n^{\frac r2}}\mu(\sigma(S^2),6) + \frac1{n^{\frac r2+1}}\big( ns^2(Y_n) + \lambda\df(Y_n) \big) - O(n^{-\frac r2-1}).
    \]
    \label{lemma:sphere_lower_bound}
\end{lemma}

\begin{proof}[Proof sketch]
    The proof is essentially the same as that of Lemma~\ref{lemma:hex_lower_bound} but we replace each (spherical) cell $D_y$ with $\varphi_y(D_y)$ where $\varphi_y$ is the gnomonic projection centered at $y$. By direct calculation in suitable coordinates, one finds 
    \[
        \abs[\bigg]{ \int_{D_y} \varrho_{S^2}(x,y)^r\, d\sigma(x) - \int_{\varphi_y(D_y)} \norm x^r \,dx} = O(n^{-\frac r2 - 2}).
    \]
    Since gnomonic projection maps convex spherical polygons to convex Euclidean polygons (with the same number of sides) we may apply the moment lemma to each $\varphi_y(D_y)$. Moreover, the total area of the projected cells is within tolerance of $\sigma(S^2)$. 
\end{proof}
The upper bound, on the other hand, poses a much more serious difficulty. Obviously, the Goldberg-Coxeter construction does not apply directly since the sphere is not made up of Euclidean equilateral triangles. The strategy instead is to carry out the Goldberg-Coxeter construction on the regular icosahedron and transport the structure to the sphere without creating too great an error. The most naive attempt, radially projecting the point configuration from the icosahedron to the sphere fails, the distortion is not sufficiently uniform across the surface. Our goal is to find a distortion of the icosahedron itself which ``cancels out'' the distortion due to radial projection. A map which achieves this in the manner required turns out to be a rather tall order. 

It's worth noting that if this program can be carried out successfully, the upper bounds will be obtained by spherical partitions very closely related (indeed, combinatorially equivalent) to the traditional \emph{Goldberg polyhedra}. (This investigation is what originally motivated this construction before we realized the comparative simplicity of the toroidal analogue.) Thus, we hope to establish not only a L\"oschian gap bound on the spherical topological defects, but more importantly, an understanding of this phenomenon in terms of a family of extremely nice spherical partitions (essentially Goldberg polyhedra) whose density as $n \to \infty$ is related to the L\"oschian numbers.

\bibliographystyle{plain}
\bibliography{references}

\end{document}